\documentclass[10pt,reqno]{amsart}
\usepackage[T1]{fontenc}
\usepackage[utf8]{inputenc}
\usepackage{lmodern}
\usepackage{float}
\usepackage[margin=1.2in]{geometry}
\usepackage{microtype}
\usepackage{mathtools,amssymb,amsthm,mathrsfs}
\usepackage{bm}
\usepackage{ytableau}
\usepackage{commath}
\usepackage{enumerate}
\usepackage{tikz}
\usetikzlibrary{shapes.geometric}
\usepackage[colorlinks=true,
linkcolor=blue!50!black,
citecolor=green!50!black,
urlcolor=blue!50!black]{hyperref}
\usepackage[nameinlink,noabbrev,capitalize]{cleveref}
\usepackage[
backend=biber,
url=false,
isbn=false,
doi=true,
backref=true,
style=numeric,
maxbibnames=99,
sorting=nyt,
giveninits=true,
]{biblatex}
\DefineBibliographyStrings{english}{
	backrefpage = {$\uparrow$},   
	backrefpages = {$\uparrow$}, 
}
\numberwithin{equation}{section}
\theoremstyle{plain}
\newtheorem{theorem}{Theorem}
\newtheorem*{theorem*}{Theorem}
\newtheorem{prop}{Proposition}
\newtheorem{lemma}{Lemma}

\newcommand{\thistheoremname}{}
\newtheorem{genericthm}{\thistheoremname}

\newtheorem*{genericthm*}{\thistheoremname}
\newenvironment{namedtheorem*}[1]{\renewcommand{\thistheoremname}{#1}\begin{genericthm*}}{\end{genericthm*}}

\newtheorem{maintheorem}{Theorem}

\crefname{maintheorem}{Theorem}{Theorems}

\theoremstyle{definition}

\newtheorem{example}{Example}

\Crefname{prop}{Proposition}{Propositions}

\allowdisplaybreaks

\DeclareMathOperator{\ad}{ad}
\DeclareMathOperator{\CT}{CT}
\newcommand{\hyper}[3][\Phi]{{}_{#2}#1_{#3}}
\let\L\relax\DeclareMathOperator{\L}{\mathcal{L}}
\DeclareMathOperator{\R}{\mathcal{R}}
\DeclareMathOperator{\M}{\mathcal{M}}
\DeclareMathOperator{\N}{\mathcal{N}}
\def\geq{\geqslant}
\def\leq{\leqslant}
\def\({\left(}
\def\){\right)}
\DeclareMathOperator{\cover}{:\!\supset}
\DeclareMathOperator{\coveredby}{\subset\!:}
\newcommand{\poch}[2]{\(#1\)_{{#2}}}
\let\=\relax\newcommand{\=}{\mathrel{\phantom{=}}}
\newcommand\Z{\mathbb{Z}}
\newcommand\Q{\mathbb{Q}}
\def\eps{\varepsilon}
\DeclareRobustCommand\qpartial[2][q]{{\partial\over\partial_{#1}#2}}
\DeclareMathOperator{\sgn}{sgn}

\newcommand*{\SetSuchThat}[1][]{} 

\newcommand*{\MvertSets}{%
	\renewcommand*\SetSuchThat[1][]{%
		\mathclose{}%
		\nonscript\;##1\vert\penalty\relpenalty\nonscript\;%
		\mathopen{}%
	}%
}
\MvertSets 
\DeclarePairedDelimiterX \Set [2] {\lbrace}{\rbrace}
{\,#1\SetSuchThat[\delimsize]#2\,}

\newcommand{\mydef}[1]{\textbf{#1}}

\title[Difference Equations for Macdonald Hypergeometric Series]{A Characterization of the Macdonald Hypergeometric Series ${}_r\Phi_s(x;q,t)$ and ${}_r\Phi_s(x,y;q,t)$ via $q$-Difference Equations}
\author{Hong Chen}
\address[H.~Chen]{Department of Mathematics, Rutgers University, NJ, US}
\email{hc813@math.rutgers.edu}
\date{\today}
\begin{document}
\begin{abstract}
	In two widely circulated manuscripts from the 1980s, I.~G.~Macdonald introduced certain multivariate hypergeometric series ${}_pF_q(x;\alpha)$ and ${}_pF_q(x,y;\alpha)$ and their $q$-analogs ${}_r\Phi_s(x;q,t)$ and ${}_r\Phi_s(x,y;q,t)$. These series are given by explicit expansions in Jack and Macdonald polynomials, and they generalize the hypergeometric functions of one and two matrix arguments from statistics.
	
	In a recent joint paper with Siddhartha Sahi, we constructed differential operators that characterize the Jack series ${}_pF_q$ thereby answering a question of Macdonald. In this paper we construct analogous $q$-difference operators that characterize the Macdonald series ${}_r\Phi_s$.
	
	More precisely, we construct three $q$-difference operators $\mathcal A^{(x,y)}$, $\mathcal B^{(x)}$, $C^{(x)}$. The equation	$\mathcal A^{(x,y)}(f(x,y))=0$ characterizes ${}_r\Phi_s(x,y;q,t)$, while the equations $\mathcal B^{(x)}(f(x))=0$ and $\mathcal C^{(x)}(f(x))=0$ each characterize ${}_r\Phi_s(x;q,t)$. These characterizations are subject to certain symmetry, boundary and stability condition.	
	In the special case of ${}_2\Phi_1(x;q,t)$, our operator $\mathcal B^{(x)}$ was previously constructed by Kaneko in 1996. 
\end{abstract}
\maketitle
\section{Introduction}\label{sec:intro}
	Hypergeometric functions serve as a unifying framework for many special functions in mathematics and physics. 
	They appear naturally as solutions to differential or difference equations and encode rich combinatorial structures.
\subsection{The univariate cases}
	\subsubsection{Hypergeometric functions}
	Hypergeometric functions were first studied by Euler in 1769, as a tool in studying the following second-order linear ODE:
	\begin{align}\label{eqn:euler}
		z(1-z)\dod[2]{F(z)}{z} +(c-(a+b+1)z)\dod{F(z)}{z} -abF(z)=0.
	\end{align}
	Later, Gauss (1812) gave the first full systematic treatment to Euler's solution of the equation, now called the \mydef{Gauss hypergeometric series}.
	
	The theory has been extended to hypergeometric series with $p$ numerator parameters and $q$ denominator parameters:
	\begin{align}\label{eqn:pfq}
		\hyper[f]{p}{q}(a_1,\dots,a_p;b_1,\dots,b_q;z) = \sum_{k=0}^\infty \frac{\poch{a_1}{k}\cdots \poch{a_p}{k}}{\poch{b_1}{k}\cdots \poch{b_q}{k}} \frac{z^k}{k!},
	\end{align}
	where $\poch{a}{k}=a(a+1)\cdots (a+k-1)$ is the \mydef{Pochhammer symbol}.
	(Our series $\hyper[f]{p}{q}$ is usually denoted by $\hyper[F]{p}{q}$ in the literature.)
	The Gauss hypergeometric series is $\hyper[f]{2}{1}(a,b;c;z)$.
	
	Generalizing Euler's \cref{eqn:euler}, the hypergeometric series $\hyper[f]{p}{q}$ is the unique power series solution of the following differential equation
	\begin{align}\label{eqn:pfq-de}
		\(z\dod{}{z}\prod_{k=1}^q \(z\dod{}{z}+b_k-1\)
		-z \prod_{k=1}^p \(z\dod{}{z}+a_k\)\)(F(z))=0,	\quad F(0)=1.
	\end{align}
	See \cite{Erd81,AAR99,AskeyBateman}.
	
	\subsubsection{Basic hypergeometric series}
	In 1846, Heine introduced a $q$-analog of the Gauss hypergeometric series, which was also later generalized to admit more parameters.
	The modern definition of the \mydef{basic hypergeometric series} is 
	\begin{align}\label{eqn:basic}
		\hyper[\phi]{r}{s}(a_1,\dots,a_r;b_1,\dots,b_s;z;q) = \sum_{k=0}^\infty \((-1)^kq^{\binom{k}{2}}\)^{s+1-r} \frac{\poch{a_1;q}{k}\cdots\poch{a_r;q}{k}}{\poch{b_1;q}{k}\cdots\poch{b_s;q}{k}} \frac{z^k}{\poch{q;q}{k}},
	\end{align}
	where $\poch{a;q}{k}=(1-a)(1-aq)\cdots(1-aq^{k-1})$ is the \mydef{$q$-Pochhammer symbol}.
	
	Since $\frac{1-q^a}{1-q}\to a$ as $q\to1$, we have $\frac{\poch{q^a;q}{k}}{(1-q)^k}\to\poch{a}{k}$, and thus
	\begin{align}
		\lim_{q\to1}\hyper[\phi]{r}{s}(q^{a_1},\dots,q^{a_r};q^{b_1},\dots,q^{b_s};(q-1)^{s+1-r}z;q) = \hyper[f]{r}{s}(a_1,\dots,a_r;b_1,\dots,b_s;z). 
	\end{align}
	
	The series $\hyper[\phi]{r}{s}(z)=\hyper[\phi]{r}{s}(a_1,\dots,a_r;b_1,\dots,b_s;z;q)$ satisfies the following $q$-difference equation:
	\begin{align}\label{eqn:q-difference}
		\Delta_1\Delta_{b_1/q}\cdots \Delta_{b_s/q}(F(z)) =z\Delta_{a_1}\cdots\Delta_{a_r}(F(q^{s+1-r}z)),
	\end{align}
	where $\Delta_aF(z)=aF(qz)-F(z)$. 
	This equation is the $q$-analog of \cref{eqn:pfq-de}; see \cite[Exercise~1.31]{GR04}. 
	
	For a very brief historical account of hypergeometric series and differential equations, we refer to the introduction of \cite{CS-Jack}. 
	For a more comprehensive introduction, see \cite{AAR99,GR04,AskeyBateman}.
\subsection{The multivariate cases}
	In the study of multivariate statistics,  Herz \cite{Herz} and Constantine \cite{Con63} introduced multivariate analogs of hypergeometric series that are associated with zonal polynomials.
	Following them, Macdonald introduced \mydef{Jack hypergeometric series} $\hyper[F]{p}{q}$ associated with Jack polynomial and \mydef{Macdonald hypergeometric series} $\hyper{r}{s}$ associated with Macdonald polynomials in his manuscripts \cite{Mac-HG,Mac-HG-2}.
	
	In what follows, fix $n\geq1$ and write $x=(x_1,\dots,x_n)$, $y=(y_1,\dots,y_n)$, $\bm1_n=(1,\dots,1)$, $\bm0_n=(0,\dots,0)$ ($n$ times), and $t^{\bm\delta_n}=(t^{n-1},t^{n-2},\dots,1)$.
	\subsubsection{Jack hypergeometric series}
	Macdonald defined the following Jack hypergeometric series in one and two alphabets:
	\begin{align}
		\hyper[F]{p}{q}(a_1,\dots,a_p;b_1,\dots,b_q;x;\alpha)	
		&=	\sum_\lambda \frac{\poch{a_1;\alpha}{\lambda}\cdots\poch{a_p;\alpha}{\lambda}}{\poch{b_1;\alpha}{\lambda}\dots\poch{b_q;\alpha}{\lambda}} \alpha^{|\lambda|} \frac{J_\lambda(x;\alpha)}{j_\lambda^{(\alpha)}},	\label{eqn:pFq-Jack-1}\\ 
		\hyper[F]{p}{q}(a_1,\dots,a_p;b_1,\dots,b_q;x,y;\alpha)	
		&=	\sum_\lambda \frac{\poch{a_1;\alpha}{\lambda}\cdots\poch{a_p;\alpha}{\lambda}}{\poch{b_1;\alpha}{\lambda}\dots\poch{b_q;\alpha}{\lambda}} \alpha^{|\lambda|} \frac{J_\lambda(x;\alpha)J_\lambda(y;\alpha)}{j_\lambda^{(\alpha)} J_\lambda(\bm1_n;\alpha)},\label{eqn:pFq-Jack-2}
	\end{align}
	where the sums run over all partitions of length at most $n$, $\poch{\cdot;\alpha}{\lambda}$ is the \mydef{$\alpha$-Pochhammer symbol}, and $J_\lambda(\cdot;\alpha)$ is the Jack polynomial.
	See \cref{sec:jack}.
	
	We shall often drop parameters and write $\hyper{p}{q}(x) = \hyper[F]{p}{q}(a_1,\dots,a_p;b_1,\dots,b_q;x;\alpha)$ and similarly for $\hyper{p}{q}(x,y)$.
	When $n=1$, the series $\hyper[F]{p}{q}$ in \cref{eqn:pFq-Jack-1} reduces to the univariate series $\hyper[f]{p}{q}$ in \cref{eqn:pfq}. 
	As generalizations of \cref{eqn:pfq-de}, differential equations for $\hyper[F]{p}{q}(x)$ and $\hyper[F]{p}{q}(x,y)$ have been studied in the zonal case  \cite{M70,CM72,Fujikoshi} and in the Jack case \cite{Mac-HG,Yan92,Kan93,BF97}, for small $p$ and $q$.
	In a recent work \cite{CS-Jack}, together with Siddhartha Sahi, we solved this question: 
	
	We found the following differential operators:
	\begin{align*}
		\L^{(\alpha)}, \quad
		\M^{(\alpha)}, \quad
		\N^{(\alpha)}, \quad
		\R^{(\alpha)}.
	\end{align*}
	The operator $\L^{(\alpha)}$ \textit{lowers} the degree of Jack polynomials by one, $\R$ \textit{raises} it by one, while $\M$ and $\N$ are \textit{eigen-operators}, acting diagonally on Jack polynomials $(J_\lambda(\alpha))$.
	
	We proved that the series $\hyper[F]{p}{q}(x,y)$ is the unique solution of the equation
	\begin{align}\label{eqn:intro-1}	
		\(\L^{(\alpha),(x)}-\R^{(\alpha),(y)}\)(F(x,y))=0,\quad F(\bm0_n,\bm0_n)=1,
	\end{align}
	where the superscripts $x$ and $y$ indicate differentiation with respect to the corresponding alphabet; the series $\hyper{p}{q}(x)$ is the unique solution of the equation 
	\begin{align}\label{eqn:intro-2}
		\(\L^{(\alpha),(x)}-\M^{(\alpha),(x)}\)(F(x))=0,	\quad F(\bm0_n)=1,
	\end{align}
	subject to certain stability condition;
	and the series $\hyper{p}{q}(x)$ is the unique solution of the equation
	\begin{align}\label{eqn:intro-3}
		\(\N^{(\alpha),(x)}-\R^{(\alpha),(x)}\)(F(x))=0,	\quad F(\bm0_n)=1.
	\end{align}
	Here, uniqueness is under the assumption that $F(x,y)$ and $F(x)$ are expressible as certain formal power series.
	
	\subsubsection{Macdonald polynomials}
	Macdonald introduced the following Macdonald hypergeometric series:
	\begin{align}
		\hyper{r}{s}(a_1,\dots,a_r;b_1,\dots,b_s;x;q,t) &= \sum_{\lambda} \frac{\poch{a_1;q,t}{\lambda}\cdots\poch{a_r;q,t}{\lambda}}{\poch{b_1;q,t}{\lambda}\dots\poch{b_s;q,t}{\lambda}} t^{n(\lambda)} \frac{J_\lambda(x;q,t)}{j_\lambda^{(q,t)}}	\label{eqn:pFq-Mac-1}\\
		\hyper{r}{s}(a_1,\dots,a_r;b_1,\dots,b_s;x,y;q,t) &= \sum_{\lambda} \frac{\poch{a_1;q,t}{\lambda}\cdots\poch{a_r;q,t}{\lambda}}{\poch{b_1;q,t}{\lambda}\dots\poch{b_s;q,t}{\lambda}} t^{n(\lambda)} \frac{J_\lambda(x;q,t)J_\lambda(y;q,t)}{j_\lambda^{(q,t)} J_\lambda(t^{\bm\delta_n};q,t)},\label{eqn:pFq-Mac-2}
	\end{align}
	where the sums run over partitions of length at most $n$, and $\poch{\cdot;q,t}{\lambda}$ is the \mydef{$(q,t)$-Pochhammer symbol} (defined below), $J_\lambda(\cdot;q,t)$ is the Macdonald polynomial.
	
	The basic hypergeometric series $\hyper[\phi]{r}{s}(z;q)$ and the Jack hypergeometric series $\hyper[F]{p}{q}(x;\alpha)$ generalize the classical hypergeometric series $\hyper[f]{p}{q}(z)$ in two completely different directions: the former concerns $q$-Pochhammer symbol and a single variable, while the latter $\alpha$-Pochhammer symbol and $n$-variate Jack polynomials.
	The Macdonald hypergeometric series $\hyper{r}{s}(x;q,t)$ then \textit{unify} the two generalizations: the $(q,t)$-Pochhammer symbol reduces to the $q$- and the $\alpha$-Pochhammer symbols and Macdonald polynomial to a single variable and Jack polynomial.
	For instance, for $n=1$, writing $x_1$ as $z$, we have 
	\begin{align*}
		\hyper{2}{1}(a,b;c;x;q,t) = \hyper[\phi]{2}{1}(a,b;c;z;q).
	\end{align*}
	As for the Jack hypergeometric series, we have 
	\begin{align*}
		\lim_{q\to1} \hyper{2}{1}(q^a,q^b;q^c;x;q,q^{1/\alpha})
		= \hyper[F]{2}{1}(a,b;c;x;\alpha).
	\end{align*}
	
	A natural question is to find $q$-difference equations that characterize the Macdonald hypergeometric series $\hyper{r}{s}(x)$ and $\hyper{r}{s}(x,y)$. In this paper, we answer this question, generalizing the results of \cite{CS-Jack}.
	
	For other important properties of the Macdonald hypergeometric series $\hyper{r}{s}$---including, but not limited to, summation formulas, transformation identities, $q$-Selberg integrals, and integral representations---we refer the reader to the aforementioned papers, as well as \cite{Kan96,BF99,War05,FW08} and the references therein.
\subsection{Our work}
	In this paper, we generalize the ideas in the Jack case \cite{CS-Jack} to the Macdonald settings, and prove analogous results.
	To be more precise, we find the following $q$-difference operators: 
	\begin{align*}
		\L=\L^{(q,t)},\quad
		\M=\M^{(q,t)},\quad
		\N=\N^{(q,t)},\quad
		\R=\R^{(q,t)}.
	\end{align*}
	The operator $\L$ \textit{lowers} the degree of Macdonald polynomials by one, $\R$ \textit{raises} it by one, while $\M$ and $\N$ are \textit{eigen-operators}, acting diagonally on Macdonald polynomials $(J_\lambda(q,t))$.
	
	We show in \cref{thm:2arg,thm:1arg-lowering,thm:1arg-raising} that the Macdonald hypergeometric series $\hyper{r}{s}(x,y)$ and $\hyper{r}{s}(x)$ are uniquely characterized by equations similar to \cref{eqn:intro-1,eqn:intro-2,eqn:intro-3}, with $(\alpha)$ replaced by $(q,t)$.
\subsection{Related results}
	We recall some results from Kaneko's paper \cite{Kan96}.
	See \cref{sec:kaneko} for a detailed discussion.
	
	In \cite{Kan96}, Kaneko independently defined and studied the series $\hyper[\widetilde\Phi]{r}{s}(x)$ and $\hyper[\widetilde\Phi]{r}{s}(x,y)$, whose definition involves an extra factor compared to Macdonald's \cref{eqn:pFq-Mac-1,eqn:pFq-Mac-2}, and coincides with Macdonald's when $r+1=s$. 
	Kaneko's main result \cite[Theorem 4.10]{Kan96} established that the series $\hyper{2}{1}(a,b;c;x;q,t)$ is the unique solution of the difference equation \cite[Eq.~(4.2)]{Kan96}.
	One can verify (see \cref{sec:kaneko}) that Kaneko's difference operator is proportional to our operator in \cref{thm:1arg-lowering}, hence we generalize Kaneko's results to arbitrary $r$ and $s$.
	
	Using the difference equation, Kaneko proved a $q$-Selberg integral formula \cite[Theorem~5.1]{Kan96}. 
	In addition, Kaneko \cite[Proposition~5.3]{Kan96} found an integral representation of $\hyper{r+1}{s+1}(x)$ in terms of $\hyper{r}{s}(x,y)$, generalizing the Jack analog proved by Yan \cite[Proposition~3.3]{Yan92}. A summation formula for $\hyper{2}{1}$ was also given in \cite[Proposition~5.4]{Kan96}.
	
	This work builds upon the ideas introduced in \cite{CS-Jack}, where this paper was announced as forthcoming. 
	During the preparation of this manuscript---which was circulated privately in October 2025 and forms part of the author's PhD thesis \cite{Chen-thesis}---we became aware of the updated preprint \cite{LWYZ-Mac} (December 2025). 
	While their work also employs the techniques of \cite{CS-Jack} to obtain overlapping results, the findings presented here were derived independently and prior to their release.
	
\subsection{Organization}
The paper is organized as follows. 
\cref{sec:pre} recalls preliminaries on partitions, Macdonald polynomials, operators, and binomial coefficients. 
\cref{sec:2arg,sec:1arg} develop the theory for two-alphabet and one-alphabet Macdonald hypergeometric series, respectively. 
Finally, \cref{sec:related} discusses connections with the Jack hypergeometric series~\cite{CS-Jack}, Kaneko's hypergeometric series~\cite{Kan96}, and the classical univariate case~\cite{GR04}.

\section{Preliminaries}\label{sec:pre}
	We shall follow most notation in \cite{CS-Jack}. 
	Throughout let $n\geq1$ be the number of variables and let $x=(x_1,\dots,x_n)$ and $y=(y_1,\dots,y_n)$. 
	We denote $\bm1_n=(1,\dots,1)$, $\bm0_n=(0,\dots,0)$ ($n$ times), 
	$\bm\delta_n=(n-1,n-2,\dots,0)$ and $t^{\bm\delta_n}=(t^{n-1},t^{n-2},\dots,1)$. 
\subsection{Partitions}
	A \mydef{partition}, of length at most $n$, is an $n$-tuple $\lambda=(\lambda_1,\dots,\lambda_n)$ of integers such that $\lambda_1\geq\cdots\geq\lambda_n\geq0$.	
	The \mydef{size} is $|\lambda|\coloneqq\lambda_1+\dots+\lambda_n$ and the \mydef{length} is $\ell(\lambda)\coloneqq\max\Set{i}{\lambda_i>0}$.
	We shall identify a partition with its \mydef{Young diagram}, namely, the set $\Set{(i,j)\in\Z^2}{1\leq j\leq \lambda_i,\,1\leq i\leq n}$.
	We denote by $(0)$ the zero partition.
	Denote by $\mathcal P_n$ the set of partitions of length at most $n$.
	
	The \mydef{conjugate} of $\lambda$, denoted by $\lambda'$, is the partition (possibly having length more than $n$) whose diagram is the transpose of the diagram of $\lambda$.
	Equivalently, we have $\lambda_j'=\#\Set{i}{\lambda_i\geq j}$ for $j\geq1$.
	
	The \mydef{arm}, \mydef{co-arm}, \mydef{leg}, and \mydef{co-leg} of $(i,j)\in\lambda$ is
	\begin{align}
		a_\lambda(i,j) \coloneqq \lambda_i-j,\quad a'_\lambda(i,j) \coloneqq j-1,\quad
		l_\lambda(i,j) \coloneqq \lambda_j'-i,\quad l_\lambda'(i,j) \coloneqq i-1.
	\end{align}
	
	We need the following partial orders on $\mathcal P_n$: 
	we say $\lambda$ \mydef{contains} $\mu$, and write $\lambda\supseteq\mu$, if $\lambda_i\geq\mu_i$ for $1\leq i\leq n$;
	$\lambda$ \mydef{covers} $\mu$, $\lambda\cover\mu$, if $\lambda\supseteq\mu$ and $|\lambda|=|\mu|+1$; 
	$\lambda$ \mydef{dominates} $\mu$, $\lambda\geq\mu$, if $|\lambda|=|\mu|$ and $\lambda_1+\dots+\lambda_i\geq\mu_1+\dots+\mu_i$ for $1\leq i\leq n$.

	Recall that the usual \mydef{$q$-Pochhammer symbol} is defined as
	\begin{align}
		\poch{u;q}{m} =\prod_{j=1}^{m} (1-uq^{j-1})= (1-u)(1-uq)\cdots (1-uq^{m-1}).	
	\end{align} 
	Define the \mydef{$(q,t)$-Pochhammer symbol} as
	\begin{align}
		\poch{u;q,t}{\lambda} \coloneqq \prod_{(i,j)\in\lambda} (1-uq^{j-1}t^{1-i})
		=\prod_{i=1}^n \poch{ut^{1-i};q}{\lambda_i}.
	\end{align}
	When $\lambda=(m)$ is a row  partition, $\poch{u;q,t}{\lambda}=\poch{u;q}{m}$.
	From now on, write $\poch{u}{\lambda}=\poch{u;q,t}{\lambda}$.
	For a tuple $\underline{a}=(a_1,\dots,a_r)$, let
	\begin{align}
		\poch{\underline{a}}{\lambda}=\poch{\underline{a};q,t}{\lambda}=\poch{a_1;q,t}{\lambda}\cdots \poch{a_r;q,t}{\lambda}=(a_1)_\lambda\cdots(a_r)_\lambda.
	\end{align}
	
	Consider the following statistics
	\begin{align}
		n(\lambda) \coloneqq \sum_{(i,j)\in\lambda}(i-1),	\quad
		\rho(\lambda) \coloneqq \sum_{(i,j)\in\lambda} q^{j-1}t^{1-i}.
	\end{align}
	When $\lambda\cover\mu$, write
	\begin{align}
		n(\lambda/\mu) \coloneqq n(\lambda)-n(\mu),\quad
		\rho(\lambda/\mu) \coloneqq \rho(\lambda)-\rho(\mu),
	\end{align}
	then one easily sees the following identity that will be crucial later:
	\begin{align}\label{eqn:poch-ratio}
		\frac{\poch{u}{\lambda}}{\poch{u}{\mu}} =1-u\rho(\lambda/\mu),	\quad\lambda\cover\mu.
	\end{align}
\subsection{Symmetric polynomials}
	Let $\Lambda_{n,\Q}$ be the algebra of symmetric polynomials in the variables $x=(x_1,\dots,x_n)$ over $\Q$.
	The monomial $m_\lambda$, elementary $e_\lambda$, power sum $p_\lambda$, Schur $s_\lambda$ are well-known symmetric polynomials indexed by partitions. See \cite{Mac15,CS-Jack}.
	
	Let $\Lambda_{n,\Q(q,t)}=\Lambda_{n,\Q}\otimes\Q(q,t)$, where $q$ and $t$ are indeterminates over $\Q$.
	Define the \mydef{$(q,t)$-Hall inner product} on $\Lambda_{n,\Q(q,t)}$ by 
	\begin{align}
		\langle p_\lambda,p_\mu\rangle _{q,t} \coloneqq \delta_{\lambda\mu} z_\lambda \prod_{i=1}^{\ell(\lambda)} \frac{1-q^{\lambda_i}}{1-t^{\lambda_i}},
	\end{align}
	where $z_\lambda\coloneqq\prod_r (r^{m_r}m_r!)$ and $m_r\coloneqq \#\Set{i}{\lambda_i=r}$ is the multiplicity of $r$ in $\lambda$.
	
	It is well-known in \cite[Chapter~VI]{Mac15} that (the monic form) \mydef{Macdonald polynomials} $P_\lambda(q,t)$ are uniquely determined by orthogonality, triangularity and normalization:
	\begin{gather}
		\langle P_\lambda(q,t),P_\mu(q,t)\rangle _{q,t} = 0,\quad	\lambda\neq\mu,	\\
		P_\lambda(q,t)=\sum_{\mu\leq\lambda}\widetilde K_{\lambda\mu}(q,t)m_\mu,	\\
		\widetilde K_{\lambda,\lambda}(q,t)=1.
	\end{gather}
	
	The expansion coefficients $\widetilde K_{\lambda\mu}$ are generalizations of the Kostka numbers $K_{\lambda\mu}$ for Schur polynomials.
	Haglund--Haiman--Loehr \cite{HHL05} gave combinatorial interpretations for $\widetilde K_{\lambda\mu}(q,t)$.
	
	Macdonald polynomials specialize to many families of symmetric polynomials. 
	On the boundary of the unit square, $q=0$, $t=0$, $q=1$ and $t=1$, Macdonald polynomials reduce to Hall--Littlewood polynomials, $q$-Whittaker polynomials, elementary and monomial symmetric polynomials, respectively.
	On the diagonal $q=t$, Macdonald polynomials become Schur polynomials.
	In addition, the Macdonald polynomial $J_\lambda(q,t)$ reduces to the Jack polynomial $J_\lambda(\alpha)$ via the limit $t=q^{1/\alpha}$, $q\to1$; see \cref{sec:jack}.
	
	Macdonald polynomials are \mydef{stable} in the sense that $m$-variate ($m\leq n$) Macdonald polynomials are specialization of $n$-variate ones:
	\begin{align}\label{eqn:mac-stable}
		P_\lambda(x_1,\dots,x_m;q,t) = P_\lambda(x_1,\dots,x_m,0,\dots,0;q,t),\quad \ell(\lambda)\leq m. 
	\end{align}
	
	Define the following \mydef{$(q,t)$-hook-length} functions:
	\begin{align}
		c_\lambda(i,j) &\coloneqq 1-q^{a_\lambda(i,j)}t^{l_\lambda(i,j)+1}=1-q^{\lambda_i-j}t^{\lambda_j'-i+1},
		\\c_\lambda'(i,j) &\coloneqq 1-q^{a_\lambda(i,j)+1}t^{l_\lambda(i,j)}= 1- q^{\lambda_i-j+1}t^{\lambda_j'-i},
	\end{align}
	and 
	\begin{align}
		c_\lambda\coloneqq \prod_{(i,j)\in\lambda} c_\lambda(i,j), \quad	
		c_\lambda'\coloneqq \prod_{(i,j)\in\lambda} c_\lambda'(i,j),\quad
		j_\lambda\coloneqq c_\lambda c_\lambda'.
	\end{align} 
	Define the \mydef{integral form} and the \mydef{dual form} Macdonald polynomials by
	\begin{align}
		J_\lambda \coloneqq c_\lambda P_\lambda,{\quad and \quad}
		J_\lambda^* \coloneqq \frac{P_\lambda}{c_\lambda'} = \frac{J_\lambda}{j_\lambda }.
	\end{align}	
	These satisfy $\langle J_\lambda,J_\mu^*\rangle _{q,t}=\delta_{\lambda\mu}$. 
	
	Let $\eps_{u,t}$ be the homomorphism $\Lambda_{n,\Q(q,t)}\to\Q(q,t)$ given by
	\begin{align}
		\eps_{u,t}(p_r) = \frac{1-u^r}{1-t^r}, \quad 1\leq r\leq n.
	\end{align}
	In particular, for $f\in\Lambda_{n,\Q(q,t)}$, $\eps_{t^n,t}(f)=f(t^{n-1},t^{n-2},\dots,1)=f(t^{\bm\delta_n})$.
	By \cite[(VI.8.8)]{Mac15}
	\begin{align}
		\eps_{t^n,t}(J_\lambda) = \prod_{(i,j)\in\lambda} (t^{i-1}-t^nq^{j-1}) = t^{n(\lambda)}\poch{t^n;q,t}{\lambda}.
	\end{align}
	The \mydef{unital form} Macdonald polynomial is defined by
	\begin{align}
		\Omega_\lambda(x;q,t) \coloneqq \frac{J_\lambda(x;q,t)}{J_\lambda(t^{\bm\delta_n};q,t)}.
	\end{align}
	Note that it depends on $n$.
\subsection{Macdonald hypergeometric series}
	We now define the $n$-variate \mydef{Macdonald hypergeometric series}:
	\begin{align}
		\hyper{r}{s}(x) = \hyper{r}{s}^{(n)}(\underline{a};\underline{b};x;q,t) &= \sum_{\lambda\in\mathcal P_n}\frac{\poch{\underline{a};q,t}{\lambda}}{\poch{\underline{b};q,t}{\lambda}} t^{n(\lambda)} J_\lambda^*(x;q,t)	\label{eqn:pFq-Mac-1-new}\\
		\hyper{r}{s}(x,y)= \hyper{r}{s}^{(n)}(\underline{a};\underline{b};x,y;q,t) &= \sum_{\lambda\in\mathcal P_n}\frac{\poch{\underline{a};q,t}{\lambda}}{\poch{\underline{b};q,t}{\lambda}} t^{n(\lambda)} \Omega_\lambda(x;q,t) J_\lambda^*(y;q,t).
		\label{eqn:pFq-Mac-2-new}
	\end{align}
	Here, $r$ and $s\geq0$ are arbitrary, and $\underline{a}=(a_1,\dots,a_r)$ and $\underline{b}=(b_1,\dots,b_s)$ are indeterminates.
	We view \cref{eqn:pFq-Mac-1-new,eqn:pFq-Mac-2-new} as formal power series. 
	We often omit some parameters when there is no confusion.
	
	By stability of Macdonald polynomials \cref{eqn:mac-stable}, Macdonald hypergeometric series are also stable:
	\begin{align}
		\hyper{r}{s}^{(m)}(\underline{a};\underline{b};x_1,\dots,x_m;q,t)
		=\hyper{r}{s}^{(n)}(\underline{a};\underline{b};x_1,\dots,x_m,0\dots,0;q,t).
	\end{align}
	Unless otherwise stated, we only work with the $n$-variate case and the superscript $n$ will be omitted.
	
	It follows directly that
	\begin{align}
		\hyper{r}{s}(\underline{a};\underline{b};x,y;q,t)=\hyper{r}{s}(\underline{a};\underline{b};y,x;q,t),	\label{eqn:sym-xy}\quad
		\hyper{r}{s}(\underline{a};\underline{b};x,t^{\bm\delta_n};q,t)=	\hyper{r}{s}(\underline{a};\underline{b};x;q,t).
	\end{align}
	
	Macdonald \cite{Mac-HG-2} proved that 
	\begin{align*}
		\hyper{0}{0}(x;q,t) = \prod_{i=1}^n \frac{1}{\poch{x_i;q}{\infty}},	\quad
		\hyper{1}{0}(a;x;q,t) = \prod_{i=1}^n \frac{\poch{ax_i;q}{\infty}}{\poch{x_i;q}{\infty}}.
	\end{align*}
	Also, since $\poch{0;q,t}{\lambda}=1$ for all partitions,  
	\begin{align}\label{eqn:0para}
		\hyper{r+1}{s}(\underline{a},0;\underline{b};x;q,t) = 
		\hyper{r}{s+1}(\underline{a};\underline{b},0;x;q,t) = 
		\hyper{r}{s}(\underline{a};\underline{b};x;q,t).
	\end{align}
	The Cauchy identity for Macdonald polynomials reads
	\begin{align}
		\hyper{1}{0}(t^n;x,y;q,t)
		=\sum_{\lambda\in\mathcal P_n} \frac{J_\lambda(x)J_\lambda(y)}{j_\lambda} = \prod_{i,j=1}^n \frac{\poch{tx_iy_j;q}{\infty}}{\poch{x_iy_j;q}{\infty}}
	\end{align}
\subsection{Operators and binomial coefficients}
	Let
	\begin{align}
		e_1'(x)=\frac{e_1(x)}{1-q} = \frac{x_1+\dots+x_n}{1-q}.
	\end{align}
	For $1\leq i\leq n$, let
	\begin{align}
		A_i(x;t) \coloneqq \prod_{j\neq i} \frac{tx_i-x_j}{x_i-x_j}.
	\end{align}
	Define the $i$-th \mydef{$q$-shift operator} $T_{q,i}=T_{q,x_i}$ and \mydef{$q$-derivative operator} $\qpartial{x_i}$ on any function $f(x_1,\dots,x_n)$ as follows:
	\begin{align}
		(T_{q,i} f)(\dots,x_i,\dots) &\coloneqq f(\dots,qx_i,\dots),\\
		\qpartial{x_i}(f) &\coloneqq \frac{1}{x_i}\frac{T_{q,i}(f)-f}{q-1},
	\end{align}
	here, ``$\dots$'' means the remaining variables are unchanged.

	Define the operators $E$ and $\square$ by
	\begin{align}
		E &\coloneqq \sum_{i=1}^n A_i(x;t)\qpartial{x_i},	\\
		\square &\coloneqq \frac{1}{t^{n-1}}\sum_{i=1}^n x_iA_i(x;t)\qpartial{x_i}.
	\end{align}
	These two operators were studied by Lassalle in \cite{Las98}. 
	(Note: our operator $E$ is Lassalle's $E_0$ and our $\square$ is Lassalle's $\frac{1}{t^{n-1}}E_1$.)
	We collect some properties below.
	\begin{lemma}[{\cite[Proposition~7]{Las98}}]
		The operators $E$ and $\square$ can be written as
		\begin{align*}
			E &= \frac{1}{q-1} \sum_{i=1}^\infty \frac{A_i(x;t)T_{q,i}-1}{x_i},	\\
			\square &=  \frac{1}{q-1}\frac{1}{t^{n-1}} \(\sum_{i=1}^\infty  A_i(x;t)T_{q,i}-\frac{1-t^n}{1-t}\).
		\end{align*}
	\end{lemma}
	\begin{lemma}\label{lem:action}
		For any pair of partitions $\lambda\cover\mu$, there exists generalized binomial coefficient $\binom{\lambda}{\mu}\in\Q(q,t)$, such that 
		\begin{align}
			e_1' \cdot J_\mu^* &= \sum_{\lambda\cover\mu} t^{n(\lambda/\mu)}\binom{\lambda}{\mu} J_{\lambda}^*,	\\
			E(\Omega_\lambda) &= \sum_{\mu\coveredby\lambda} \binom{\lambda}{\mu} \Omega_{\mu},	\\
			\square (J_\lambda) &= \rho(\lambda) J_\lambda.
		\end{align}
	\end{lemma}
	\begin{proof}
		See Theorem 3, Theorem 5, and Equation (7.2) in \cite{Las98}.
		Note that our binomial coefficient $\binom{\lambda}{\mu}$ is denoted as $\binom{\lambda}{\mu}_{q,t}$ in \cite{Las98}, and our $J^*_\lambda$ and $\Omega_\lambda$ as $J_\lambda^\sharp$ and $J_\lambda^*$, respectively.
	\end{proof}
	
	Recall that the \mydef{adjoint action} is $\ad_{A}(B) \coloneqq [A,B] = AB-BA$, where $A$ and $B$ are operators. Define, recursively, $\ad_{A}^r(B) \coloneqq \ad_A^{r-1}([A,B])$, for $r\geq1$. By convention, let $\ad_{A}^0(B)=B$.
	
	The following proposition is motivated by the Jack case \cite{Mac-HG,CS-Jack}. 
	\begin{prop}\label{prop:ad}
		We have
		\begin{align}\label{eqn:square_e1}
			[\square,e_1'](J_\mu^*)
			&=	 \sum_{\lambda\cover\mu} \rho(\lambda/\mu) t^{n(\lambda/\mu)} \binom{\lambda}{\mu} J_\lambda^*,	\\
			[E,\square](\Omega_\lambda) 
			&= \sum_{\mu\coveredby\lambda} \rho(\lambda/\mu)\binom{\lambda}{\mu} \Omega_{\mu}.
		\end{align}
		More generally, for $l\geq0$, 
		\begin{align}
			\(\ad_{\square}^l(e_1')\)(J_\mu^*)
			&=	 \sum_{\lambda\cover\mu} \rho(\lambda/\mu)^l t^{n(\lambda/\mu)} \binom{\lambda}{\mu} J_\lambda^*,	\label{eqn:adl-e1}\\
			\(\ad_{-\square}^l(E)\)(\Omega_\lambda) 
			&=	\sum_{\mu\coveredby\lambda} \rho(\lambda/\mu)^l \binom{\lambda}{\mu} \Omega_{\mu}.
		\end{align}
	\end{prop}
	\begin{proof}
		For simplicity, we shall only prove \cref{eqn:adl-e1}. 
		Induction on $l$.
		The base case $l=0$ is in \cref{lem:action}.
		For the inductive step, we have
		\begin{align*}
			\(\ad_{\square}^{l+1}(e_1')\)(J_\mu^*)
			&=	[\square,\ad_{\square}^{l}(e_1')](J_\mu^*)
			=	\square \(\(\ad_{\square}^{l}(e_1')\)(J_\mu^*)\) -\(\ad_{\square}^{l}(e_1')\)\(\square(J_\mu^*)\)
			\\&=	\square\(\sum_{\lambda\cover\mu} \rho(\lambda/\mu)^l t^{n(\lambda/\mu)} \binom{\lambda}{\mu} J_\lambda^*\) -\(\ad_{\square}^{l}(e_1')\)\(\rho(\mu)J_\mu^*\)
			\\&=	\sum_{\lambda\cover\mu} \rho(\lambda)\rho(\lambda/\mu)^l t^{n(\lambda/\mu)} \binom{\lambda}{\mu} J_\lambda^* -\rho(\mu)\sum_{\lambda\cover\mu} \rho(\lambda/\mu)^l t^{n(\lambda/\mu)} \binom{\lambda}{\mu} J_\lambda^*
			\\&=	\sum_{\lambda\cover\mu} \rho(\lambda/\mu)^{l+1} t^{n(\lambda/\mu)} \binom{\lambda}{\mu} J_\lambda^*.\qedhere
		\end{align*}
	\end{proof}
	We remark that \cite[Propositions 8 and 9]{Las98} gave a relation equivalent to \cref{eqn:square_e1}. 
\subsection{Macdonald operator}
	Recall that the Macdonald operator $D(u)$ is defined as (see \cite[\textsection VI.3]{Mac15} and \cite[\textsection 5.3]{Noumi})
	\begin{align}
		D(u)=D(u;q,t) \coloneqq \frac{1}{V(x)} \sum_{w\in S_n} \sgn(w) x^{w\bm\delta_n} \prod_{i=1}^n (1+ut^{(w\bm\delta_n)_i} T_{q,i}),
	\end{align}
	which acts diagonally on $(J_\lambda)$ by
	\begin{align}
		D(u)(J_\lambda) = \prod_{i=1}^n(1+uq^{\lambda_i}t^{n-i})\cdot J_\lambda.
	\end{align}
	Define the operators $D_l$, $0\leq l\leq n$, via the expansion:
	\begin{align}
		D(u) = \sum_{l=0}^n u^lD_l,
	\end{align}
	then $D_l$ commutes pairwise since each acts diagonally on $(J_\lambda)$ by
	\begin{align}
		D_l(J_\lambda) = e_l(\overline\lambda)\cdot J_\lambda,
	\end{align}
	where $\overline\lambda_i=q^{\lambda_i}t^{n-i}$, $i=1,\dots,n$.
	
	The operator $D_1$ is closely related to the operator $\square$ since
	\begin{align*}
		D_1=\sum_{i=1}^n A_i(x;t)T_{q,i},
	\end{align*} 
\section{Two alphabets}\label{sec:2arg}
	In this section, we consider
	\begin{align*}
		\hyper{r}{s} = \hyper{r}{s}(\underline{a};\underline{b};x,y;q,t) &= \sum_{\lambda}\frac{\poch{\underline{a}}{\lambda}}{\poch{\underline{b}}{\lambda}} t^{n(\lambda)} \Omega_\lambda(x) J_\lambda^*(y).
	\end{align*}
	Define
	\begin{align}
		\hyper[\L]{}{s}^{(x)}(\underline{b}) &\coloneqq \sum_{l=0}^s (-1)^le_{l}(\underline{b})\ad_{-\square^{(x)}}^l(E^{(x)}) = \(\prod_{k=1}^s (1-b_k\ad_{-\square^{(x)}})\)(E^{(x)}),	\label{eqn:Lsx}\\
		\hyper[\R]{r}{}^{(y)}(\underline{a}) &\coloneqq \sum_{l=0}^r (-1)^le_{l}(\underline{a})\ad_{\square^{(y)}}^l(e_1'(y))= \(\prod_{k=1}^r (1-a_k\ad_{\square^{(y)}})\)(e_1'(y)),\label{eqn:rRy}
	\end{align}
	where $\square^{(x)}$ and $\square^{(y)}$ indicate the operator $\square$ with respect to the $x$ and $y$ variables, respectively, and similarly for $E^{(x)}$ and $e_1'(y)$.
	
	Then by \cref{prop:ad}, we have
	\begin{align}
		\hyper[\L]{}{s}^{(x)}(\underline{b})(\Omega_\lambda(x)) &=
		\sum_{\mu\coveredby\lambda} \sum_{l=0}^s (-1)^le_l(\underline{b})\rho(\lambda/\mu)^l \binom{\lambda}{\mu} \Omega_\mu(x)	\notag
		\\&=\sum_{\mu\coveredby\lambda} \prod_{k=1}^s \(1-b_k\rho(\lambda/\mu)\) \binom{\lambda}{\mu} \Omega_\mu(x), 	\\\intertext{and similarly,}
		\hyper[\R]{r}{}^{(y)}(\underline{a})(J_\mu^*(y)) &=
		\sum_{\lambda\cover\mu} \prod_{k=1}^r \(1-a_k\rho(\lambda/\mu)\) t^{n(\lambda/\mu)} \binom{\lambda}{\mu} J_\lambda^*(y).
	\end{align}
	
	Let $\mathscr F^{(x,y)}$ be the space over $\mathbb Q(q,t)$ of formal power series in the form 
	\begin{align*}
		F(x,y)=\sum_{}C_{a_1,\dots,a_n,b_1,\dots,b_n}(q,t) x_1^{a_1}\cdots x_n^{a_n} y_1^{b_1}\cdots y_n^{b_n},
	\end{align*}
	where the sum is over $a_1,\dots,a_n,b_1,\dots,b_n\geq0$ and $C_{a_1,\dots,a_n,b_1,\dots,b_n}(q,t)\in\mathbb Q(q,t)$, and $F(x,y)$ is symmetric in $x$ and in $y$ separately.
	The hypergeometric operators $\hyper[\L]{}{s}^{(x)}(\underline{b})$ and $\hyper[\R]{r}{}^{(y)}(\underline{a})$ and the Macdonald operators $D^{(x)}(u)$ and $D^{(y)}(u)$ are linear maps on the space $\mathscr F^{(x,y)}$. 
	
	Define $\mathscr F^{(x,y)}_{D}$ as the subspace of $\mathscr F^{(x,y)}$ that consists of formula power series in the form
	\begin{align}\label{eqn:Fxy}
		F(x,y) = \sum_{\lambda} C_\lambda(q,t) t^{n(\lambda)} \frac{J_\lambda(x)J_\lambda(y)}{j_\lambda J_\lambda(t^{\bm\delta_n})},
	\end{align}
	where $C_\lambda(q,t)\in\Q(q,t)$.
	\begin{prop}
		The following are equivalent for $F(x,y)\in\mathscr F^{(x,y)}$: 
		\begin{enumerate}[(1)]
			\item $F(x,y)\in\mathscr F^{(x,y)}_D$.
			\item $F(x,y)$ is in the kernel of $D^{(x)}(u)-D^{(y)}(u)$.
			\item $F(x,y)$ and $D^{(x)}(t)(F(x,y))$ are $\tau$-invariant, where $\tau:\mathscr F^{(x,y)}\to \mathscr F^{(x,y)}$ is the involution of interchanging $x$ and $y$.
		\end{enumerate}
	\end{prop}
	\begin{proof}
		The implication (1)$\implies$(3) is obvious.
		
		For (3)$\implies$(2), we note that $\tau D^{(x)}=D^{(y)}\tau$, and hence
		\begin{align*}
			D^{(x)}F -D^{(y)}F = D^{(x)}F -D^{(y)}\tau F = D^{(x)}F -\tau D^{(x)}F=0.
		\end{align*}
		
		We now prove that (2)$\implies$(1).
		Since $(J_\lambda(y))$ forms a basis, we may write $$F(x,y) = \sum_\lambda A_\lambda(x)J_\lambda(y)\in\mathscr F^{(x,y)}$$ for some symmetric formal power series $A_\lambda(x)$.
		Applying $D^{(x)}(u)-D^{(y)}(u)$, we get 
		\begin{align*}
			\sum_\lambda \(D^{(x)}(u)(A_\lambda(x))-\prod_{i=1}^n(1+uq^{\lambda_i}t^{n-i})\cdot A_\lambda(x)\)J_\lambda(y)=0,
		\end{align*}
		which implies 
		\begin{align*}
			D^{(x)}(u)(A_\lambda(x)) = \prod_{i=1}^n(1+uq^{\lambda_i}t^{n-i})\cdot A_\lambda(x)
		\end{align*}
		for all $\lambda$.
		By the well-known fact that the Macdonald operator $D(u)$ characterizes Macdonald polynomials, we conclude that $A_\lambda(x)$ is a scalar multiple of $J_\lambda(x)$.
	\end{proof}

	\begin{maintheorem}\label{thm:2arg}
		The hypergeometric series $\hyper{r}{s}(\underline{a};\underline{b};x,y;q,t)$ is the unique solution in $\mathscr F^{(x,y)}_D$ of the equation 
		\begin{align}\label{eqn:Fxy-LR}
			\(\hyper[\L]{}{s}^{(x)}(\underline{b}) -\hyper[\R]{r}{}^{(y)}(\underline{a})\) (F(x,y))=0,
		\end{align}
		subject to the initial condition that $F(\bm0_n,\bm0_n)=1$, i.e., $C_{(0)}(q,t)=1$.
	\end{maintheorem}
	\begin{proof}
		Let $F(x,y)$ in \cref{eqn:Fxy} be a solution of \cref{eqn:Fxy-LR}. 
		Then we have
		\begin{align*}
			\hyper[\L]{}{s}^{(x)}(\underline{b})(F(x,y))
			&=	\sum_{\lambda} C_\lambda(q,t) t^{n(\lambda)} \sum_{\mu\coveredby\lambda} \prod_{k=1}^s \(1-b_k\rho(\lambda/\mu)\) \binom{\lambda}{\mu} \Omega_\mu(x) J_\lambda^*(y)	
			\\&=	\sum_{\mu} \sum_{\lambda\cover\mu} C_\lambda(q,t) \prod_{k=1}^s \(1-b_k\rho(\lambda/\mu)\) t^{n(\lambda)} \binom{\lambda}{\mu} \Omega_\mu(x) J_\lambda^*(y),	\\
			\hyper[\R]{r}{}^{(y)}(\underline{a})(F(x,y))
			&=	\sum_{\mu} \sum_{\lambda\cover\mu} C_\mu(q,t) \prod_{k=1}^r \(1-a_k\rho(\lambda/\mu)\) t^{n(\lambda)} \binom{\lambda}{\mu} \Omega_\mu(x) J_\lambda^*(y).
		\end{align*}
		By comparing the coefficients of $\Omega_\mu(x) J_\lambda^*(y)$ for each pair $\lambda\cover\mu$ (note that in this case $\binom{\lambda}{\mu}\neq0$), we have 
		\begin{align}\label{eqn:rec}
			C_\lambda(q,t) \prod_{k=1}^s \(1-b_k\rho(\lambda/\mu)\) 
			=
			C_\mu(q,t) \prod_{k=1}^r \(1-a_k\rho(\lambda/\mu)\).
		\end{align}
		This recursion and the initial condition determine $(C_\lambda(q,t))$ uniquely.
		
		By \cref{eqn:poch-ratio}, it is obvious that $C_\lambda=\frac{\poch{\underline{a}}{\lambda}}{\poch{\underline{b}}{\lambda}}$ satisfies the recursion and the initial condition, and hence $\hyper{r}{s}(\underline{a};\underline{b};x,y;q,t)$ is the unique solution of \cref{eqn:Fxy-LR}.
	\end{proof}
\subsection{Interchanging the variables}
	By the symmetry between $x$ and $y$ in \cref{eqn:sym-xy}, we may interchange the $x$ and $y$ variables in $\hyper[\L]{}{s}^{(x)}$ and $\hyper[\R]{r}{}^{(y)}$ and define:
	\begin{align}
		\hyper[\L]{}{s}^{(y)}(\underline{b}) &\coloneqq \sum_{l=0}^s (-1)^le_{l}(\underline{b})\ad_{-\square^{(y)}}^l(E^{(y)}) = \(\prod_{k=1}^s (1-b_k\ad_{-\square^{(y)}})\)(E^{(y)}),	\label{eqn:Lsy}\\
		\hyper[\R]{r}{}^{(x)}(\underline{a}) &\coloneqq \sum_{l=0}^r (-1)^le_{l}(\underline{a})\ad_{\square^{(x)}}^l(e_1'(x))= \(\prod_{k=1}^r (1-a_k\ad_{\square^{(x)}})\)(e_1'(x)).\label{eqn:rRx}
	\end{align}
	Then we have 
	\begin{namedtheorem*}{Theorem~A$'$}\label{thm:1'}
		\cref{thm:2arg} holds for the equation
		\begin{align}
			\(\hyper[\L]{}{s}^{(y)}(\underline{b}) -\hyper[\R]{r}{}^{(x)}(\underline{a})\) (F(x,y))=0.
		\end{align}
	\end{namedtheorem*}
\section{One alphabet}\label{sec:1arg}
	In this section, we consider 
	\begin{align*}
		\hyper{r}{s} = \hyper{r}{s}(\underline{a};\underline{b};x;q,t) &= \sum_{\lambda}\frac{\poch{\underline{a}}{\lambda}}{\poch{\underline{b}}{\lambda}} t^{n(\lambda)} J_\lambda^*(x).
	\end{align*}
\subsection{The lowering operator}
	Using the lowering operator $\hyper[\L]{}{s}$ defined in \cref{eqn:Lsx}, we have
	\begin{align*}
		\hyper[\L]{}{s}(\hyper{r}{s})
		&=	\sum_{\lambda}\frac{\poch{\underline{a}}{\lambda}}{\poch{\underline{b}}{\lambda}} t^{n(\lambda)} J_\lambda^*(t^{\bm\delta_n})\hyper[\L]{}{s}^{}(\Omega_\lambda)	
		\\&=	\sum_{\lambda}\frac{\poch{\underline{a}}{\lambda}}{\poch{\underline{b}}{\lambda}} t^{n(\lambda)} J_\lambda^*(t^{\bm\delta_n}) \sum_{\mu\coveredby\lambda} \prod_{k=1}^s \(1-b_k\rho(\lambda/\mu)\) \binom{\lambda}{\mu}\Omega_\mu
		\\&=	\sum_{\mu}\frac{\poch{\underline{a}}{\mu}}{\poch{\underline{b}}{\mu}} t^{n(\mu)} \(\sum_{\lambda\cover\mu} \prod_{k=1}^r \(1-a_k\rho(\lambda/\mu)\) t^{n(\lambda/\mu)} \binom{\lambda}{\mu} \frac{J_\lambda^*(t^{\bm\delta_n})}{J_\mu^*(t^{\bm\delta_n})}\)\cdot J_\mu^*,
	\end{align*}
	In other words, we need to find eigen-operators that have 
	\begin{align}
		G_{l,n}(\mu)\coloneqq\sum_{\lambda\cover\mu} \rho(\lambda/\mu)^l \cdot t^{n(\lambda/\mu)}\binom{\lambda}{\mu} \frac{J_\lambda^*(t^{\bm\delta_n})}{J_\mu^*(t^{\bm\delta_n})}
	\end{align}
	as eigenvalues when acting on $(J_\mu)$.
	
	In this subsection only, for $i=1,\dots,n$, let $z_i=z_i(\mu)\coloneqq q^{\mu_i}t^{1-i}=t^{1-n}\overline \mu_i$.
	\begin{lemma}\label{lem:bino}
		Let $\lambda=\mu+\eps_{i_0}$ be a partition for some $i_0$, then
		$\rho(\lambda/\mu)=z_{i_0}$ and
		\begin{align}
			t^{n(\lambda/\mu)}\binom{\lambda}{\mu} \frac{J_\lambda^*(t^{\bm\delta_n})}{J_\mu^*(t^{\bm\delta_n})}
			=\frac{A_{i_0}(z;t)}{1-q}
			=\frac{1}{1-q}\prod_{i\neq i_0}\frac{tz_{i_0}-z_i}{z_{i_0}-z_i}.
		\end{align}
	\end{lemma}
	\begin{proof}
		Let $s_0=(i_0,\lambda_{i_0})$ be $\lambda/\mu$, and $R$ and $C$ be the set of other boxes in the row and the column of $s_0$ in $\lambda$.
		By \cite[Proposition~4.3(3)]{CS24},
		\begin{align*}
			\binom{\lambda}{\mu} = t^{-n(\lambda/\mu)} \prod_{s\in C}\frac{c_\lambda(s)}{c_\mu(s)} \prod_{s\in R}\frac{c_\lambda'(s)}{c_\mu'(s)}.
		\end{align*}
		By definition, 
		\begin{align*}
			\frac{j_\mu}{j_\lambda} = \prod_{s\in R\cup C} \frac{c_\mu(s)c_\mu'(s)}{c_\lambda(s)c_\lambda'(s)}\cdot \frac{1}{c_\lambda(s_0)c_\lambda'(s_0)},
		\end{align*}
		\begin{align*}
			\frac{J_\lambda(t^{\bm\delta_n})}{J_\mu(t^{\bm\delta_n})}
			=t^{n(\lambda/\mu)} (1-t^n\rho(\lambda/\mu))
			=t^{n(\lambda/\mu)} (1-t^nz_{i_0}),
		\end{align*}
		hence
		\begin{align*}
			t^{n(\lambda/\mu)}\binom{\lambda}{\mu} \frac{J_\lambda^*(t^{\bm\delta_n})}{J_\mu^*(t^{\bm\delta_n})} 
			= t^{n(\lambda/\mu)}\prod_{s\in C}\frac{c_\mu'(s)}{c_\lambda'(s)} \prod_{s\in R}\frac{c_\mu(s)}{c_\lambda(s)} \cdot \frac{1-t^nz_{i_0}}{c_\lambda(s_0)c_\lambda'(s_0)}.
		\end{align*}
		Now, since
		\begin{align*}
			\prod_{s\in\text{row }i_0} c_\lambda(s) =\poch{t^{n-i+1};q}{\lambda_{i_0}} \cdot \prod_{i=i_0+1}^n \frac{\poch{t^{i-i_0};q}{\lambda_{i_0}-\lambda_i}}{\poch{t^{i-i_0+1};q}{\lambda_{i_0}-\lambda_i}},
		\end{align*}
		we have
		\begin{align*}
			\prod_{s\in R}\frac{c_\mu(s)}{c_\lambda(s)} \frac{1}{c_\lambda(s_0)}
			=\frac{1}{1-t^{n-i_0+1}q^{\mu_i}} \prod_{i=i_0+1}^n \frac{1-t^{i-i_0+1}q^{\mu_{i_0}-\mu_i}}{1-t^{i-i_0}q^{\mu_{i_0}-\mu_i}}
			=\frac{1}{1-t^nz_{i_0}} \prod_{i=i_0+1}^n \frac{tz_{i_0}-z_i}{z_{i_0}-z_i}.
		\end{align*}
		
		As for the column, we have 
		\begin{align*}
			\prod_{s\in C}\frac{c_\mu'(s)}{c_\lambda'(s)}
			=	\prod_{i=1}^{i_0}\frac{1-q^{\mu_i-\mu_{i_0}}t^{i_0-i-1}}{1-q^{\mu_i-\mu_{i_0}}t^{i_0-i}}
			=	\frac{1}{t^{i_0-1}}\prod_{i=1}^{i_0-1} \frac{tz_{i_0}-z_i}{z_{i_0}-z_i}.
		\end{align*}
		The conclusion follows by combining the identities above and note that $c_\lambda'(s_0)=1-q$ and $n(\lambda/\mu)=i_0-1$.
	\end{proof}
	Note that when $\mu+\eps_{i_0}$ is not a partition, namely, when $\mu_{i_0-1}=\mu_{i_0}$, we have $A_{i_0}(z;t)=0$.
	
	Define the generating function
	\begin{align}
		G_{n}(\mu;u) = \sum_{l=0}^\infty G_{l,n}(\mu) u^l. 
	\end{align}
	\begin{prop}
		The generating function $G_{n}(\mu;u)$ can be given by
		\begin{align}
			G_n(\mu;u) = \frac{1}{(1-q)(1-t)}\(1-\prod_{i=1}^n \frac{t-uz_i}{1-uz_i}\).
		\end{align}
	\end{prop}
	\begin{proof}
		By definition and \cref{lem:bino}, we have 
		\begin{align*}
			G_{n}(\mu;u) = \sum_{\lambda\cover\mu} \frac{1}{1-u\rho(\lambda/\mu)} \cdot t^{n(\lambda/\mu)}\binom{\lambda}{\mu} \frac{J_\lambda^*(t^{\bm\delta_n})}{J_\mu^*(t^{\bm\delta_n})}
			= \frac{1}{1-q} \sum_{i_0=1}^n \frac{1}{1-uz_{i_0}} \prod_{i\neq i_0}\frac{tz_{i_0}-z_i}{z_{i_0}-z_i}.
		\end{align*}
		It is not hard to see that this is the partial fraction decomposition of the desired identity.
	\end{proof}
	
	Recall that the Macdonald operator $D(u)$ acts on $(J_\mu)$ by
	\begin{align*}
		D(u)(J_\mu) = \prod_{i=1}^n (1+ut^{n-1}z_i) \cdot J_\mu,
	\end{align*}
	then we have
	\begin{align*}
		\frac{1}{(1-q)(1-t)} \(1-t^n\frac{D(-ut^{-n})}{D(-ut^{1-n})}\) (J_\mu) 
		&=	\frac{1}{(1-q)(1-t)} \(1-\prod_{i=1}^n\frac{t-uz_i}{1-uz_i}\)\cdot J_\mu 
		\\&=	G_n(\mu;u)\cdot J_\mu.
	\end{align*}
	Now, define the operator $\mathcal G_n(u)$ by
	\begin{align}
		\mathcal G_n(u) &\coloneqq \frac{1}{(1-q)(1-t)} \(1-t^n\frac{D(-ut^{-n})}{D(-ut^{1-n})}\)	\notag
		\\&=\frac{1}{(1-q)(1-t)} -\frac{t^n}{(1-q)(1-t)} \sum_{l=0}^n (-ut^{-n})^l D_l\cdot \sum_{k=1}^\infty\(\sum_{l=1}^n (-ut^{1-n})^l D_l\)^k,
	\end{align}
	and define the operators $\mathcal G_{l,n}$ via the expansion
	\begin{align}
		\mathcal G_n \eqqcolon \sum_{l=0}^\infty \mathcal G_{l,n} u^l.
	\end{align}
	Then we have
	\begin{align}
		\mathcal G_n(u)(J_\mu) = G_{n}(\mu;u)\cdot J_\mu,\text{\quad and \quad}\mathcal G_{l,n}(J_\mu) = G_{l,n}(\mu)\cdot J_\mu.
	\end{align}
	\begin{example}\label{ex:G}
		The first few operators $\mathcal G_{l,n}$ are
		\begin{align*}
			\mathcal G_{0,n} &= \frac{1}{1-q}\frac{1-t^n}{1-t},	\\
			\mathcal G_{1,n} &= \frac{D_1}{1-q},	\\
			\mathcal G_{2,n} &= \frac{tD_1^2-(1+t)D_2}{(1-q)t^n},	\\
			\mathcal G_{3,n} &= \frac{t^2D_1^3-(2t^2+t)D_1D_2+(1+t+t^2)D_3}{(1-q)t^{2n}}.	
		\end{align*}
	\end{example}
	Define the operator 
	\begin{align}\label{eqn:rM}
		\hyper[\M] {r}{}(\underline{a})
		\coloneqq \sum_{l=0}^r (-1)^l e_l(\underline{a})\mathcal G_{l,n}
		=\CT\( \prod_{k=1}^r(1-a_ku^{-1})\cdot \mathcal G_{n}(u) \),
	\end{align}
	where $\CT$ denotes the constant term of the Laurent series in $u$ near $u=0$.
	Then $\hyper[\M]{r}{}(\underline{a})$ acts diagonally on $(J_\mu)$ as
	\begin{align}
		\hyper[\M]{r}{}(\underline{a})(J_\mu)
		&=	\sum_{l=0}^r (-1)^l e_l(\underline{a})G_{l,n}(\mu)\cdot J_\mu	\notag
		\\&=	\sum_{\lambda\cover\mu}\prod_{k=1}^r \(1-a_k\rho(\lambda/\mu)\) t^{n(\lambda/\mu)} \binom{\lambda}{\mu} \frac{J_\lambda^*(t^{\bm\delta_n})}{J_\mu^*(t^{\bm\delta_n})} \cdot J_\mu.
	\end{align}
	
	Now, we add a superscript $(n)$ in $\hyper[\L^{(n)}]{}{s}(\underline{b})$ and $\hyper[\M^{(n)}]{r}{}(\underline{a})$ to indicate that the operators are in $n$ variables $(x_1,\dots,x_n)$. For $1\leq m\leq n$, $\hyper[\L^{(m)}]{}{s}(\underline{b})$ and $\hyper[\M^{(m)}]{r}{}(\underline{a})$ will be the corresponding operators in the variables $(x_1,\dots,x_m)$.
	
	Let $\mathscr F^{(x_1,\dots,x_m)}$ be the space over $\Q(q,t)$ of formal power series in the form
	\begin{align}\label{eqn:Fx-m}
		F(x_1,\dots,x_m) = \sum_{\lambda\in\mathcal P_m} C_\lambda(q,t) t^{n(\lambda)} J_\lambda^*(x_1,\dots,x_m),
	\end{align}
	where each $C_\lambda(q,t)$ is in $\Q(q,t)$.
	Then $\hyper[\L^{(m)}]{}{s}(\underline{b})$ and $\hyper[\M^{(m)}]{r}{}(\underline{a})$ act on the space $\mathscr F^{(x_1,\dots,x_m)}$.
	Write $\mathscr F^{(x)}$ for $\mathscr F^{(x_1,\dots,x_n)}$.
	
	\begin{maintheorem}\label{thm:1arg-lowering}
		The hypergeometric series $\hyper{r}{s}(\underline{a};\underline{b};x_1,\dots,x_n;q,t)$ is the unique solution in $\mathscr{F}^{(x)}$ of the equation
		\begin{align}\label{eqn:ML-F}
			\(\hyper[\L^{(n)}]{}{s}(\underline{b}) - \hyper[\M^{(n)}]{r}{}(\underline{a})\) (F(x_1,\dots,x_n))=0, 
		\end{align}
		subject to the initial condition that $F(\bm0_n)=1$, and the stability condition that for each $1\leq m\leq n-1$, $\hyper{r}{s}(\underline{a};\underline{b};x_1,\dots,x_m,0,\dots,0;q,t)$ is a solution of the equation
		\begin{align}\label{eqn:ML-F-m}
			\(\hyper[\L^{(m)}]{}{s}(\underline{b}) - \hyper[\M^{(m)}]{r}{}(\underline{a})\) (F(x_1,\dots,x_m)) = 0.
		\end{align}
	\end{maintheorem}
	The stability condition is inspired by \cite[Theorem 4.10]{Kan96}.
	\begin{proof}
		For each $1\leq m\leq n$, let \cref{eqn:Fx-m} be a solution of \cref{eqn:ML-F-m}, then we have
		\begin{align*}
			&\=	\hyper[\L^{(m)}]{}{s}(\underline{b})(F(x_1,\dots,x_m))
			\\&=	\sum_{\lambda\in\mathcal P_m} C_\lambda(q,t) t^{n(\lambda)} \sum_{\mu\coveredby\lambda} \prod_{k=1}^s \(1-b_k\rho(\lambda/\mu)\) \binom{\lambda}{\mu} \frac{J_\lambda^*(t^{\bm\delta_m})}{J_\mu^*(t^{\bm\delta_m})} J_\mu^*(x_1,\dots,x_m)
			\\&=	\sum_{\mu\in\mathcal P_m} t^{n(\mu)} \sum_{\lambda\cover\mu} C_\lambda(q,t) \prod_{k=1}^s \(1-b_k\rho(\lambda/\mu)\) t^{n(\lambda/\mu)}\binom{\lambda}{\mu} \frac{J_\lambda^*(t^{\bm\delta_m})}{J_\mu^*(t^{\bm\delta_m})} J_\mu^*(x_1,\dots,x_m),
			\intertext{and}
			&\=	(\hyper[\M^{(m)}]{r}{}(\underline{a}))(F(x_1,\dots,x_m))
			\\&=	\sum_{\mu}C_\mu(q,t) t^{n(\mu)} \sum_{\lambda\cover\mu} \prod_{k=1}^r \(1-a_k\rho(\lambda/\mu)\) t^{n(\lambda/\mu)} \binom{\lambda}{\mu} \frac{J_\lambda^*(t^{\bm\delta_m})}{J_\mu^*(t^{\bm\delta_m})} J_\mu^*(x_1,\dots,x_m).
		\end{align*}
		Comparing the coefficients of $J_\mu^*(x_1,\dots,x_m)$, we have
		\begin{equation}\label{eqn:rec-up}
			\begin{split}
				&\=	\sum_{\lambda\cover\mu} C_\lambda(q,t) \prod_{k=1}^s \(1-b_k\rho(\lambda/\mu)\) t^{n(\lambda/\mu)}\binom{\lambda}{\mu} \frac{J_\lambda^*(t^{\bm\delta_m})}{J_\mu^*(t^{\bm\delta_m})} 
				\\&=C_\mu(q,t) \sum_{\lambda\cover\mu} \prod_{k=1}^r \(1-a_k\rho(\lambda/\mu)\) t^{n(\lambda/\mu)} \binom{\lambda}{\mu} \frac{J_\lambda^*(t^{\bm\delta_m})}{J_\mu^*(t^{\bm\delta_m})}.
			\end{split}
		\end{equation}
		By \cref{eqn:poch-ratio}, $C_\lambda(q,t)=\frac{\poch{\underline{a}}{\lambda}}{\poch{\underline{b}}{\lambda}}$ satisfies \cref{eqn:rec-up} and the initial condition, hence $\hyper{r}{s}(\underline{a};\underline{b};x;q,t)$ is a solution of \cref{eqn:ML-F}.
		
		Now we prove the uniqueness.
		Note that 
		\begin{align*}
			\frac{J_\lambda^*(t^{\bm\delta_m})}{J_\mu^*(t^{\bm\delta_m})} = \frac{j_\mu}{j_\lambda} t^{n(\lambda/\mu)} (1-t^m\rho(\lambda/\mu))
		\end{align*}
		is the only term involving $m$ in \cref{eqn:rec-up}. 
		Since \cref{eqn:rec-up} holds for all $1\leq m\leq n$, we can get two equations out of each \cref{eqn:rec-up}:
		\begin{align}
			\sum_{\lambda\cover\mu} C_\lambda(q,t) K_{\underline{b}}(\lambda,\mu)
			&=C_\mu(q,t) \sum_{\lambda\cover\mu} K_{\underline{a}}(\lambda,\mu), \label{eqn:I}	\\
			\sum_{\lambda\cover\mu} C_\lambda(q,t) \rho(\lambda/\mu)K_{\underline{b}}(\lambda,\mu)
			&=C_\mu(q,t) \sum_{\lambda\cover\mu} \rho(\lambda/\mu)K_{\underline{a}}(\lambda,\mu), \label{eqn:II}
		\end{align}
		where
		\begin{align*}
			K_{\underline{b}}(\lambda,\mu) &= \prod_{k=1}^s \(1-b_k\rho(\lambda/\mu)\) t^{2\cdot n(\lambda/\mu)}\binom{\lambda}{\mu} \frac{j_\mu}{j_\lambda},	\\
			K_{\underline{a}}(\lambda,\mu) &= \prod_{k=1}^r \(1-a_k\rho(\lambda/\mu)\) t^{2\cdot n(\lambda/\mu)}\binom{\lambda}{\mu} \frac{j_\mu}{j_\lambda}.
		\end{align*}
		We claim that \cref{eqn:I,eqn:II} and the initial condition determine $(C_\lambda(q,t))$ completely. 
		The following is similar to the proof in \cite[Theorem~B]{CS-Jack}.
		
		When $\mu=(0)$, $C_{(0)}(q,t)=1$, and \cref{eqn:I,eqn:II} involve one unknown $C_{(1)}(q,t)$.
		Then we have 
		\begin{align*}
			C_{(1)}(q,t)=\frac{\prod_{k=1}^r \(1-a_k\rho((1)/(0))\)}{\prod_{k=1}^s \(1-b_k\rho((1)/(0))\)} = \frac{\poch{\underline{a}}{(1)}}{\poch{\underline{b}}{(1)}}.
		\end{align*}
		
		Let $P_n(d)$ be the number of partitions of size $d$ and length at most $n$.
		By induction, assume that all $C_\mu(q,t)$ are known for $|\mu|\leq d-1$.
		Now, we have $P_n(d)$ many unknowns $C_\lambda(q,t)$, and $2P_n(d-1)$ many equations.
		Since $2P_n(d-1)\geq P_n(d)$, this is an over-determined linear system of equations. 
		
		For a partition $\mu$, let $S(\mu)$ be the set of partitions covering $\mu$.
		Let $>_{L}$ be the (reverse) lexicographical total order of the set of partitions of length at most $n$ and of size $d$, that is, $\lambda^1>_{L}\lambda^2$ if and only if the first nonzero entry of $(\lambda_i^1-\lambda^2_i)_i$ is positive.
		
		For each $d\geq1$, let $\mu^1>_{L}\dots>_{L}\mu^{M}$ be all partitions of length at most $n$ and size $d$, where $M=P_n(d)$. 
		Then $\lambda^1=(d)$ and $S(\lambda^1)=\{(d+1),(d,1)\}$.
		We claim that for each $2\leq m\leq M$, the set-theoretic difference $S(\mu^m)\setminus\bigcup_{i=1}^{m-1} S(\mu^{i})$ has cardinality at most two.
		In fact, for each $\mu^m$, there are at most $\ell(\mu^m)+1$ partitions covering $\mu^m$, namely $\mu^m+\eps_i$, for $i=1,\dots,\ell(\mu^m)+1$.
		For $i=1,\dots,\ell(\mu^m)-1$ such that $\mu^m+\eps_i$ is a partition, $\mu^m+\eps_i$ is in $\bigcup_{i=1}^{m-1} S(\mu^{i})$. 
		This is because $\mu^m+\eps_i-\eps_{\ell(\mu^m)}$ is a partition that  precedes $\mu^{m}$ and $\mu^m+\eps_i\in S(\mu^m+\eps_i-\eps_{\ell(\mu^m)})$.
		We are left with at most two partitions that cover $\mu^m$, namely $\mu^m+\eps_{\ell(\mu^m)}$ and $\mu^m+\eps_{\ell(\mu^m)+1}$.
		
		Now by induction on $m$, suppose all $C_\lambda(q,t)$'s are known for $\lambda\in\bigcup_{i=1}^{m-1} S(\mu^{i})$, then \cref{eqn:I,eqn:II} indexed by $\mu^{m}$ involve at most two new unknowns. 
		If there is no new unknown, there is nothing to be done; if there is only one new unknown, it can be determined by \cref{eqn:I} along; now, assume the two new unknowns are $C_{\lambda^{l_1}}(q,t)$ and $C_{\lambda^{l_2}}(q,t)$. 
		\cref{eqn:I,eqn:II} are in the form:
		\begin{align*}
			\begin{pmatrix}
				K_{\underline{b}}(\lambda^{l_1},\mu^m)&K_{\underline{b}}(\lambda^{l_2},\mu^m)	\\
				\rho(\lambda^{l_1}/\mu^m)K_{\underline{b}}(\lambda^{l_1},\mu^m)&\rho(\lambda^{l_2}/\mu^m)K_{\underline{b}}(\lambda^{l_2},\mu^m)
			\end{pmatrix}
			\begin{pmatrix}
				C_{\lambda^{l_1}}(q,t)\\C_{\lambda^{l_2}}(q,t)
			\end{pmatrix}=
			\begin{pmatrix}
				*\\ *
			\end{pmatrix}
		\end{align*}
		where each $*$ involves $C_\lambda(q,t)$ for $\lambda\in \bigcup_{i=1}^{m-1} S(\mu^{i})$. 
		Now, $C_{\lambda^{l_1}}(q,t)$ and $C_{\lambda^{l_2}}(q,t)$ are also determined, since the determinant $K_{\underline{b}}(\lambda^{l_1},\mu^m)K_{\underline{b}}(\lambda^{l_2},\mu^m)(\rho(\lambda^{l_2}/\mu^m)-\rho(\lambda^{l_1}/\mu^m))$ is nonzero.
	\end{proof}
	In the proof above, the extra stability condition is needed to determine the coefficients $C_\lambda$ from \cref{eqn:rec-up}. This is because \cref{eqn:rec-up} proceeds ``in the wrong direction'', which is in turn due to the lowering operator $\L$ that lowers the degree of $J_\lambda$. 
	If we use the raising operator $\R$ in the first place, an easier recursion \cref{eqn:rec-down} would show up and the uniqueness follows immediately, as explained in the next subsection.
\subsection{The raising operator}
	In this subsection, we start with the raising operator $\hyper[\R]{r}{}$ defined in \cref{eqn:rRx}. 
	We have
	\begin{align*}
		\hyper[\R]{r}{}(\hyper{r}{s})
		&=	\sum_{\mu}\frac{\poch{\underline{a}}{\mu}}{\poch{\underline{b}}{\mu}} t^{n(\mu)} \sum_{\lambda\cover\mu} \prod_{k=1}^r \(1-a_k\rho(\lambda/\mu)\) t^{n(\lambda/\mu)} \binom{\lambda}{\mu} J_\lambda^*
		\\&=	\sum_{\lambda}\frac{\poch{\underline{a}}{\lambda}}{\poch{\underline{b}}{\lambda}} t^{n(\lambda)} \sum_{\mu\coveredby\lambda} \prod_{k=1}^s \(1-b_k\rho(\lambda/\mu)\) \binom{\lambda}{\mu} J_\lambda^*.
	\end{align*}
	Let
	\begin{align}
		H_{l}(\lambda) \coloneqq \sum_{\mu\coveredby\lambda} \rho(\lambda/\mu)^l \binom{\lambda}{\mu}.
	\end{align}
	Our goal is to find eigen-operators $\mathcal H_l$ that have $H_{l}(\lambda)$ as eigenvalues when acting on $(J_\lambda)$.
	
	Define the generating function
	\begin{align}
		H(\lambda;u) = \sum_{l=0}^\infty H_l(\lambda) u^l.
	\end{align}
	In this subsection only, let $z_i=q^{\lambda_i}t^{1-i}$ for $i=1,\dots,n$ (different from the previous definition).
	\begin{lemma}\label{lem:bino2}
		Let $\mu=\lambda-\eps_{i_0}$ be a partition for some $i_0$, then 
		$\rho(\lambda/\mu)=z_i/q$ and 
		\begin{align}
			\binom{\lambda}{\mu} = \frac{1-t^{n-1}z_{i_0}}{1-q} A_{i_0}(z;1/t) = \frac{1-t^{n-1}z_{i_0}}{1-q} \prod_{i\neq i_0} \frac{z_{i_0}/t-z_{i}}{z_{i_0}-z_i}.
		\end{align}
	\end{lemma}
	\begin{proof}
		Let $s_0=(i_0,\lambda_{i_0})$ be $\lambda/\mu$, and $R$ and $C$ be the set of other boxes in the row and the column of $s_0$ in $\lambda$.
		By \cite[Proposition~4.3(3)]{CS24},
		\begin{align*}
			\binom{\lambda}{\mu} = t^{-n(\lambda/\mu)} \prod_{s\in C}\frac{c_\lambda(s)}{c_\mu(s)} \prod_{s\in R}\frac{c_\lambda'(s)}{c_\mu'(s)}.
		\end{align*}
		For the row, note that
		\begin{align*}
			\prod_{s\in\text{row }i_0} c_\lambda'(s) = \poch{qt^{n-i_0};q}{\lambda_{i_0}} \prod_{i=i_0+1}^n \frac{\poch{qt^{i-i_0-1};q}{\lambda_{i_0}-\lambda_i}}{\poch{qt^{i-i_0};q}{\lambda_{i_0}-\lambda_i}},
		\end{align*}
		hence
		\begin{align*}
			\prod_{s\in R}\frac{c_\lambda'(s)}{c_\mu'(s)} 
			&= \frac{1}{c_\lambda'(s_0)}\frac{\prod_{s\in\text{row }i_0} c_\lambda'(s)}{\prod_{s\in\text{row }i_0} c_\mu'(s)}
			=\frac{1-q^{\lambda_{i_0}}t^{n-i_0}}{1-q} \prod_{i=i_0+1}^n \frac{1-q^{\lambda_{i_0}-\lambda_i}t^{i-i_0-1}}{1-q^{\lambda_{i_0}-\lambda_i}t^{i-i_0}}	\\
			&=\frac{1-t^{n-1}z_{i_0}}{1-q} \prod_{i=i_0+1}^n \frac{z_{i_0}/t-z_{i}}{z_{i_0}-z_i}.
		\end{align*}
		For the column, we have
		\begin{align*}
			\prod_{s\in C}\frac{c_\lambda(s)}{c_\mu(s)} 
			&=	\prod_{i=1}^{i_0-1}\frac{1-q^{\lambda_i-\lambda_{i_0}}t^{i_0-i+1}}{1-q^{\lambda_i-\lambda_{i_0}}t^{i_0-i}}
			=t^{i_0-1} \prod_{i=i_0+1}^n \frac{z_{i_0}/t-z_{i}}{z_{i_0}-z_i}.
		\end{align*}
		The claim follows by combining the identities.
	\end{proof}
	\begin{prop}
		The generating function $H(\lambda;u)$ is
		\begin{align}\label{eqn:H}
			H(\lambda;u) = \frac{t}{(1-q)(1-t)} \frac{u-qt^{n-1}}{u}\(\prod_{i=1}^n\frac{1/t-uz_i/q}{1-uz_i/q}-1\)+\frac{q}{1-q}\frac{1-t^n}{1-t}\frac{1}{u}.
		\end{align}
	\end{prop}	
	\begin{proof}
		By definition and \cref{lem:bino2}, 
		\begin{align*}
			H(\lambda;u)=\sum_{\mu\coveredby\lambda} \frac{1}{1-u\rho(\lambda/\mu)} \binom{\lambda}{\mu}
			=\frac{1}{1-q} \sum_{i_0=1}^n \frac{1-t^{n-1}z_{i_0}}{1-uz_{i_0}/q} \prod_{i\neq i_0}\frac{z_{i_0}/t-z_{i}}{z_{i_0}-z_i}.
		\end{align*}
		Expand the following as partial fractions:
		\begin{align*}
			\frac{t}{(1-q)(1-t)} \frac{u-qt^{n-1}}{u}\(\prod_{i=1}^n\frac{1/t-uz_i/q}{1-uz_i/q}-1\) = \frac{a_0}{u}+\sum_{i=1}^n \frac{a_i}{1-uz_i/q}.
		\end{align*}
		(There is no constant term, by taking the limit $u\to\infty$.)
		It is not hard to find that 
		\begin{align*}
			a_0&= -\frac{q}{1-q}\frac{1-t^n}{1-t},	\\
			a_{i_0}&= \frac{1}{1-q} (1-t^{n-1}z_{i_0})  \prod_{i\neq i_0}\frac{z_{i_0}/t-z_{i}}{z_{i_0}-z_i}, \quad i_0=1,\dots,n,
		\end{align*}
		and \cref{eqn:H} follows.
	\end{proof}	
	Recall that the Macdonald operator acts on $(J_\lambda)$ by
	\begin{align*}
		D(u)(J_\lambda) = \prod_{i=1}^n (1+ut^{n-1}z_i) \cdot J_\lambda.
	\end{align*}
	Define the operator $\mathcal H(u)$ and define the operators $\mathcal H_l$ via expansion by
	\begin{align}
		\mathcal H(u)
		&\coloneqq \frac{t}{(1-q)(1-t)} \frac{u-qt^{n-1}}{u}\(t^{-n}\frac{D(-u/(qt^{n-2}))}{D(-u/(qt^{n-1}))}-1\)+\frac{q}{1-q}\frac{1-t^n}{1-t}\frac{1}{u}	\\
		&\eqqcolon \sum_{l=0}^\infty \mathcal H_l u^l,
	\end{align}
	then 
	\begin{align}
		\mathcal H(u)(J_\lambda) = H(\lambda;u)\cdot J_\lambda,\quad
		\mathcal H_l(J_\lambda) = H_l(\lambda)\cdot J_\lambda.
	\end{align}
	\begin{example}
		The first few operators $\mathcal H_l$ are
		\begin{align*}
			\mathcal H_0 &= -\frac{1}{1-q}\frac{1}{t^{n-1}} \(D_1-\frac{1-t^n}{1-t}\)=\square,	\\
			\mathcal H_1 &= \frac{1}{1-q}\frac{1}{qt^{2n-2}}\((t+1)D_2-D_1^2+D_1\),	\\
			\mathcal H_2 &= \frac{1}{1-q}\frac{1}{q^2t^{3n-3}}\(-(t^2+t+1)D_3 +(t+2)D_2D_1-D_1^3 -(t+1)D_2+D_1^2\).
		\end{align*}
		Note that the Macdonald operators $D_l$ are dependent on $n$. However, (the eigenvalues of) the operators $\mathcal H_l$ are stable with respect to $n$.
		For Macdonald operators in infinitely many variables, see \cite{NS14}.
	\end{example}
	Now, define 
	\begin{align}
		\hyper[\mathcal N]{}{s}(\underline{b}) \coloneqq \sum_{l=0}^s (-1)^le_l(\underline{b})\mathcal H_l = \CT\(\prod_{k=1}^s(1-b_ku^{-1})\cdot \mathcal H(u)\),
	\end{align}
	where $\CT$ denotes the constant term of the Laurent series in $u$ near $u=0$.
	Then $\hyper[\mathcal N]{}{s}(\underline{b})$ acts diagonally on $(J_\lambda)$ by
	\begin{align}
		\hyper[\mathcal N]{}{s}(\underline{b})(J_\lambda) 
		=\sum_{l=0}^s (-1)^le_l(\underline{b}) H_l(\lambda)
		=\sum_{\mu\coveredby\lambda} \prod_{k=1}^s \(1-b_k\rho(\lambda/\mu)\) \binom{\lambda}{\mu} \cdot J_\lambda.
	\end{align}
	\begin{maintheorem}\label{thm:1arg-raising}
		The hypergeometric series  $\hyper{r}{s}(\underline{a};\underline{b};x;q,t)$ is the unique solution in $\mathscr F^{(x)}$ of the equation
		\begin{align}\label{eqn:RN-F}
			\(\hyper[\N]{}{s}(\underline{b}) - \hyper[\R]{r}{}(\underline{a})\) (F(x))=0,
		\end{align}
		subject to the condition that $F(\bm0_n)=1$, i.e., $C_{(0)}(q,t)=1$.
	\end{maintheorem}
	\begin{proof}
		Let \cref{eqn:Fx-m} (with $m=n$) be a solution of \cref{eqn:RN-F}. 
		Then 
		\begin{align*}
			(\hyper[\R]{r}{}(\underline{a}))(F(x))
			&=	\sum_\mu C_\mu(q,t) t^{n(\mu)} \sum_{\lambda\cover\mu} \prod_{k=1}^r(1-a_k\rho(\lambda/\mu)) t^{n(\lambda/\mu)} \binom{\lambda}{\mu} J_\lambda^*
			\\&=	\sum_\lambda t^{n(\lambda)} \sum_{\mu\coveredby\lambda} C_\mu(q,t)  \prod_{k=1}^r(1-a_k\rho(\lambda/\mu)) \binom{\lambda}{\mu} J_\lambda^*
		\end{align*}
		\begin{align*}
			\hyper[\mathcal N]{}{s}(\underline{b})(F(x))
			&=	\sum_\lambda C_\lambda(q,t) t^{n(\lambda)} \sum_{\mu\coveredby\lambda} \prod_{k=1}^s \(1-b_k\rho(\lambda/\mu)\) \binom{\lambda}{\mu} J_\lambda^*,
		\end{align*}
		
		Comparing the coefficients of $J_\lambda^*(x)$, we have
		\begin{equation}\label{eqn:rec-down}
			\sum_{\mu\coveredby\lambda} C_\mu(q,t)  \prod_{k=1}^r(1-a_k\rho(\lambda/\mu)) \binom{\lambda}{\mu}
			=	C_\lambda(q,t) \sum_{\mu\coveredby\lambda} \prod_{k=1}^s \(1-b_k\rho(\lambda/\mu)\) \binom{\lambda}{\mu}.
		\end{equation}
		The recursion and the initial condition determine $(C_\lambda(q,t)))$ uniquely.
		
		It is evident that $C_\lambda(q,t)=\frac{\poch{\underline{a}}{\lambda}}{\poch{\underline{b}}{\lambda}}$ satisfies the recursion and the initial condition, hence $\hyper{r}{s}(\underline{a};\underline{b};x;q,t)$ is the unique solution to \cref{eqn:RN-F}.
	\end{proof}
\section{Related results}\label{sec:related}
\subsection{Jack polynomials and hypergeometric series}\label{sec:jack}
	Jack polynomials $(J_\lambda(x;\alpha))$ can be obtained from Macdonald polynomials $J_\lambda(x;q,t)$ via the classical limit $t=q^{1/\alpha}$, $q\to1$:
	\begin{align*}
		\lim_{q\to1} \frac{J_\lambda(x;q,q^{1/\alpha})}{(1-q)^{|\lambda|}} = J_\lambda(x;\alpha),	\\
		\lim_{q\to1} \frac{J_\lambda^*(x;q,q^{1/\alpha})}{(1-q)^{-|\lambda|}} = J_\lambda^*(x;\alpha),	\\
		\lim_{q\to1} \Omega_\lambda(x;q,q^{1/\alpha}) = \Omega_\lambda(x;\alpha),	
	\end{align*}
	and the $\alpha$-Pochhammer symbol from the $(q,t)$-Pochhammer symbol
	\begin{align*}
		\lim_{q\to1} \frac{\poch{u;q,q^{1/\alpha}}{\lambda}}{(1-q)^{|\lambda|}} =  \poch{u;\alpha}{\lambda}
	\end{align*}
	
	The Jack hypergeometric series is defined as
	\begin{align*}
		\hyper[F]{r}{s}(\underline{a};\underline{b};x;\alpha)
		=\sum_{\lambda\in\mathcal P_n} \frac{\poch{\underline{a};\alpha}{\lambda}}{\poch{\underline{b};\alpha}{\lambda}} \alpha^{|\lambda|} J_\lambda^*(x;\alpha),
	\end{align*}
	which can be derived by
	\begin{align}
		\lim_{q\to1} \hyper{r}{s}(q^{\underline{a}};q^{\underline{b}};(1-q)^{s+1-r}x;q,q^{1/\alpha})
		= \hyper[F]{r}{s}(\underline{a};\underline{b};x;\alpha),
	\end{align}
	where $q^{\underline{a}}=(q^{a_1},\dots,q^{a_r})$ and similarly for $q^{\underline{b}}$ and $cx=(cx_1,\dots,cx_n)$ for $c=(1-q)^{s+1-r}$.
	
	As for operators, however, the operators in this paper do not degenerate to the corresponding operators in \cite{CS-Jack} in general.
	Formally, the $q$-derivative degenerates to the usual derivative
	\begin{align*}
		\lim_{q\to1}\qpartial{x_i} = \lim_{q\to1}\frac{q^{x_i\partial_i}-1}{q-1}\frac{1}{x_i} = \partial_i,
	\end{align*}
	and the operator $E$ degenerates to the differential operator denoted by $E_1$ in \cite{CS-Jack}
	\begin{align*}
		\lim_{\substack{t=q^{1/\alpha}\\q\to1}}　E
		=\lim_{\substack{t=q^{1/\alpha}\\q\to1}}　\sum_{i=1}^n A_i(x;t)\qpartial{x_i} = \sum_{i=1}^n \partial_i.
	\end{align*}
	However, the operator $\square^{(q,t)}$ do not degenerate to the Laplace--Beltrami operator $\square^{(\alpha)}$, instead, it degenerates to the Euler operator, denoted by $E_2$ in \cite{CS-Jack}, 
	\begin{align*}
		\lim_{\substack{t=q^{1/\alpha}\\q\to1}} \square^{(q,t)} =\lim_{\substack{t=q^{1/\alpha}\\q\to1}} \frac{1}{t^{n-1}}\sum_i x_iA_i(x;t)\qpartial{x_i}
		=\sum_i x_i\partial_i.
	\end{align*}
	In conclusion, the operators $\L^{(q,t)},\M^{(q,t)},\N^{(q,t)},\R^{(q,t)}$ in this paper do not degenerate to those $\L^{(\alpha)},\M^{(\alpha)},\N^{(\alpha)},\R^{(\alpha)}$ in \cite{CS-Jack}.
\subsection{Kaneko's hypergeometric series}\label{sec:kaneko}
	In \cite[Eqs.~(3.10),~(5.3)]{Kan96}, Kaneko, independently form Macdonald, defined the following $q$-hypergeometric series in one and two alphabets $x=(x_1,\dots,x_n)$ and $y=(y_1,\dots,y_n)$.
	We now recall them using our notation:
	\begin{align}
		\hyper[\widetilde\Phi]{r}{s}(\underline{a};\underline{b};x;q,t) &= \sum_{\lambda\in\mathcal P_n}  \((-1)^{|\lambda|} q^{n(\lambda')} t^{-n(\lambda)}\)^{s+1-r}  \frac{\poch{\underline{a};q,t}{\lambda}}{\poch{\underline{b};q,t}{\lambda}} t^{n(\lambda)} \frac{J_\lambda(x;q,t)}{j_\lambda},	\label{eqn:Kaneko-1}\\
		\hyper[\widetilde\Phi]{r}{s}(\underline{a};\underline{b};x,y;q,t) &= \sum_{\lambda\in\mathcal P_n}  \((-1)^{|\lambda|} q^{n(\lambda')} t^{-n(\lambda)}\)^{s+1-r}  \frac{\poch{\underline{a};q,t}{\lambda}}{\poch{\underline{b};q,t}{\lambda}} t^{n(\lambda)} \frac{J_\lambda(x;q,t)J_\lambda(y;q,t)}{j_\lambda J_\lambda(t^{\bm\delta_n};q,t)}.\label{eqn:Kaneko-2}
	\end{align}
	(Note: $\hyper[\widetilde\Phi]{r}{s}$ was denoted by $\hyper{r}{s}^{(q,t)}$ in \cite{Kan96}.)
	Compared to Macdonald's definition \cref{eqn:pFq-Mac-1-new,eqn:pFq-Mac-2-new}, Kaneko's definition involves an extra factor $\((-1)^{|\lambda|} q^{n(\lambda')} t^{-n(\lambda)}\)^{s+1-r}.$
	Hence, the two definitions coincide when $s+1=r$.
	
	Kaneko's definition admits the following limit relation: 
	\begin{align}
		\lim_{a_{r+1}\to\infty}\hyper[\widetilde\Phi]{r+1}{s}(\underline{a},a_{r+1};\underline{b};x/a_{r+1};q,t) = \hyper[\widetilde\Phi]{r}{s}(\underline{a};\underline{b};x;q,t).
	\end{align}
	Macdonald's definition has a simpler formula in \cref{eqn:0para}.
	
	Since 
	\begin{align*}
		\lim_{u\to\infty} \frac{\poch{u;q,t}{\lambda}}{u^{|\lambda|}} = (-1)^{|\lambda|} q^{n(\lambda')} t^{-n(\lambda)},
	\end{align*}
	when $r\leq s$, Kaneko's definition can be derived from Macdonald's via the limit
	\begin{align}
		\hyper[\widetilde\Phi]{r}{s}(\underline{a};\underline{b};x;q,t) = \lim_{\substack{u_1\to\infty\\\dots\\u_{s+1-r}\to\infty}} \hyper{s+1}{s}(\underline{a},u_1,\dots,u_{s+1-r};\underline{b};\frac{x}{u_1\cdots u_{s+1-r}};q,t).
	\end{align}
	Termwise, Kaneko's definition and Macdonald's are related by
	\begin{align}
		\hyper[\widetilde\Phi]{r}{s}(\underline{a};\underline{b};x;q,t) 
		&= \hyper{r}{s} (\underline{a}^{-1};\underline{b}^{-1};\frac{a_1\cdots a_r}{q b_1\cdots b_s}x;q^{-1},t^{-1})	\label{eqn:rel-1}\\
		\hyper[\widetilde\Phi]{r}{s}(\underline{a};\underline{b};x,y;q,t) 
		&= \hyper{r}{s} (\underline{a}^{-1};\underline{b}^{-1};\frac{a_1\cdots a_r}{qt^{n-1}b_1\cdots b_s}x,y;q^{-1},t^{-1}),	\label{eqn:rel-2}
		\\\intertext{and conversely,}
		\hyper{r}{s}(\underline{a};\underline{b};x;q,t) 
		&= \hyper[\widetilde\Phi]{r}{s} (\underline{a}^{-1};\underline{b}^{-1};\frac{a_1\cdots a_r}{qb_1\cdots b_s}x;q^{-1},t^{-1}),	\label{eqn:rel-3}\\
		\hyper{r}{s}(\underline{a};\underline{b};x,y;q,t) 
		&= \hyper[\widetilde\Phi]{r}{s}(\underline{a}^{-1};\underline{b}^{-1};\frac{a_1\cdots a_r}{qt^{n-1}b_1\cdots b_s}x,y;q^{-1},t^{-1}),	\label{eqn:rel-4}
	\end{align}
	where $\underline{a}^{-1}=(a_1^{-1},\dots,a_r^{-1})$ and similarly for $\underline{b}^{-1}$.
	One can see that the operations are involutive.
	
	It should be noted that these relations are only formally valid. 
	For instance, Macdonald's $\hyper{0}{0}$ and Kaneko's $\hyper[\widetilde\Phi]{0}{0}$ are as follows:
	\begin{align}
		\hyper{0}{0}(x;q,t) &= \sum_\lambda t^{n(\lambda)} \frac{J_\lambda(x;q,t)}{j_\lambda(q,t)} = \prod_{i=1}^n(x_i;q)_\infty^{-1},	\\
		\hyper[\widetilde\Phi]{0}{0}(x;q,t) &= \sum_\lambda (-1)^{|\lambda|} q^{n(\lambda')} \frac{J_\lambda(x;q,t)}{j_\lambda(q,t)}
		= \prod_{i=1}^n(x_i;q)_\infty.
	\end{align}
	Formally, for each term, we have 
	\begin{align*}
		(-1)^{|\lambda|} q^{n(\lambda')} \frac{J_\lambda(x;q,t)}{j_\lambda(q,t)}
		\mapsto (-1)^{|\lambda|} q^{-n(\lambda')} \frac{J_\lambda(x/q;1/q,1/t)}{j_\lambda(1/q,1/t)} 
		= t^{n(\lambda)}\frac{J_\lambda(x;q,t)}{ j_\lambda(q,t)},
	\end{align*}
	but 
	\begin{align*}
		\prod_{i=1}^n (x/q;1/q)_\infty \neq \prod_{i=1}^n(x_i;q)_\infty^{-1}.
	\end{align*}
	In fact, the LHS converges for $q>1$, while the RHS converges for $0<q<1$.
	
	We now discuss operators. In \cite{Kan96}, the number of variables is $m$, and we change it into $n$.
	The operators in \cite{Kan96} can be ``translated'' into our notation as the following table shows.  
	\begin{table}[h!]
		\centering
		\renewcommand{\arraystretch}{2.5} 
		\begin{tabular}{c|c}
			\hline
			\cite{Kan96} & This paper \\
			\hline
			$\varepsilon$ & $E$ \\
			$D_1^{(q,t)}$ & $\displaystyle t^{n-1}\square = -\frac{D_1-[n]_t}{1-q}$ \\
			$D_2^{(q,t)}$ & $\displaystyle \frac{(1+t)D_2-t[n]_t [n-1]_t}{(1-q)(1-t^2)}$ \\
			\hline
		\end{tabular}
		\vspace{\baselineskip}
		\caption{Translation of the  operators in \cite{Kan96}}
		\label{tab}
	\end{table}
	In the table and below, $[n]_t \coloneqq \frac{1-t^n}{1-t}$ is the $t$-number.
	
	As explained in \cref{sec:intro}, the main result in \cite{Kan96} is a $q$-difference equation characterization of $\hyper{2}{1}(x)$, which we now recall.
	\begin{theorem}[{\cite[Theorem 4.10]{Kan96}}]\label{thm:kaneko}
		The hypergeometric series $\hyper{2}{1}(a,b;c;x_1,\dots,x_n;q,t)$ is the unique solution of the equation $L^{(n)}(S(x_1,\dots,x_n))=0$ (defined below), subject to the following conditions:
		\begin{enumerate}
			\item $S(x_1,\dots,x_n)$ is symmetric in $x_1,\dots,x_n$;
			\item $S(x_1,\dots,x_n)$ is analytic at the origin with $S(\bm0_n)=1$;
			\item $S(x_1,\dots,x_m,0,\dots,0)$ is a solution of $L^{(m)}(S(x_1,\dots,x_m))=0$, for every $m\leq n$.
		\end{enumerate}
	\end{theorem}
	The operator $L^{(n)}$ was given in \cite[Eq.(4.2)]{Kan96}, which, in our notation, is 
	\begin{align*}
		L^{(n)}
		&=	c\frac{t^{n-1}}{1-q}[\square,E] 
		-ab\(\(\frac{D_1-[n]_t}{1-q}\)^2-\frac{1-t^2}{t(1-q)}\(\frac{(1+t)D_2-t[n]_t [n-1]_t}{(1-q)(1-t^2)}\)\)
		\\&\=	+\frac{t^{n-1}}{1-q}E
		-\frac{1}{1-q}\(2ab[n]_t-(a+b)t^{n-1}\) \(\frac{D_1-[n]_t}{1-q}\)
		-(1-a)(1-b)\frac{t^{n-1}[n]_t}{(1-q)^2}.
	\end{align*}

	We then have
	\begin{align*}
		-\frac{1-q}{t^{n-1}}L_n &= 
		\(ab\frac{tD_1^2-(1+t) D_2}{(1-q)t^{n}} 
		-(a+b)\frac{ D_1}{1-q}
		+\frac{[n]_t}{1-q}\)
		-\(E+c [\square, E]\).
	\end{align*}
	This is equal to our $\hyper[\M]{2}{}(a,b)-\hyper[\L]{}{1}(c)$; see \cref{eqn:rM,eqn:Lsx,ex:G}.
	In other words, \cref{thm:kaneko} is a special case of \cref{thm:1arg-lowering} when $(r,s)=(2,1)$.
	
	More generally, we can extend our \cref{thm:2arg,thm:1arg-lowering,thm:1arg-raising} into Kaneko's series $\hyper[\widetilde\Phi]{r}{s}$ using the relations \cref{eqn:rel-1,eqn:rel-2,eqn:rel-3,eqn:rel-4}.
	
	Write $\mathcal I$ for the operator that inverts the parameters $\underline{a}$, $\underline{b}$, $q$ and $t$, and $\mathcal S_c$ for scaling the $x$ variables:
	\begin{align*}
		\mathcal S_c (f(x)) = f(cx).
	\end{align*}
	Then \cref{eqn:rel-1,eqn:rel-2} read
	\begin{align*}
		\hyper[\widetilde\Phi]{r}{s}(\underline{a};\underline{b};x;q,t) 
		=	\mathcal S_c (\mathcal I(\hyper{r}{s}(\underline{a};\underline{b};x;q,t))),	\quad
		\hyper[\widetilde\Phi]{r}{s}(\underline{a};\underline{b};x,y;q,t) 
		=	\mathcal S_d (\mathcal I(\hyper{r}{s}(\underline{a};\underline{b};x,y;q,t))),
	\end{align*}
	where 
	\begin{align*}
		c=\frac{a_1\cdots a_r}{qb_1\dots b_s}, \quad d=\frac{a_1\cdots a_r}{qt^{n-1}b_1\dots b_s}.
	\end{align*}
	
	If an operator $\mathcal D$ characterizes a series $f$, then $\mathcal S_{c}\mathcal I\mathcal D\mathcal I^{-1}\mathcal S_c^{-1}$ characterizes $g=\mathcal S_c (\mathcal I(f))$.

	Note that conjugation of $\mathcal I$ on an operator $\mathcal D$ simply invert the parameters in $\mathcal D$.
	Denote by $\widetilde\L$, $\widetilde\M$, $\widetilde\N$, $\widetilde\R$ the operators $\L,\M,\N,\R$ with parameters inverted, respectively. 
	Also note that the operator $\square$ and the Macdonald operators $D_1,\dots,D_n$ are invariant under conjugation of $\mathcal S_{c}$, and 
	\begin{align*}
		\mathcal S_{c}E \mathcal S_c^{-1} = \frac{1}{c}E, \quad \mathcal S_{c}e_1' \mathcal S_c^{-1} = ce_1'.
	\end{align*}
	Then we have
	\begin{align}
		\mathcal S_{c}\mathcal I\L\mathcal I^{-1}\mathcal S_{c}^{-1} = \frac{1}{c}\widetilde\L,\quad
		\mathcal S_{c}\mathcal I\M\mathcal I^{-1}\mathcal S_{c}^{-1} = \widetilde\M,\quad
		\mathcal S_{c}\mathcal I\N\mathcal I^{-1}\mathcal S_{c}^{-1} = \widetilde\N,\quad
		\mathcal S_{c}\mathcal I\R\mathcal I^{-1}\mathcal S_{c}^{-1} = c\widetilde\R.
	\end{align}
	We conclude the following:
	\begin{namedtheorem*}{Theorem $\widetilde A$}
		Let $d=\frac{a_1\cdots a_r}{qt^{n-1}b_1\dots b_s}$ be as above. 
		The hypergeometric series $\hyper[\widetilde\Phi]{r}{s}(\underline{a};\underline{b};x,y;q,t)$ is the unique solution in $\mathscr F^{(x,y)}_D$ of the equation 
		\begin{align}
			\(\frac{1}{d}\cdot\hyper[\widetilde\L]{}{s}^{(x)}(\underline{b}) - \hyper[\widetilde\R]{r}{}^{(y)}(\underline{a})\) (F(x,y))=0,
		\end{align}
		subject to the initial condition that $F(\bm0_n,\bm0_n)=1$, i.e., $C_{(0)}(q,t)=1$.
	\end{namedtheorem*}
	\begin{namedtheorem*}{Theorem $\widetilde B$}
		Let $c=\frac{a_1\cdots a_r}{qb_1\dots b_s}$ be as above.
		The hypergeometric series $\hyper[\widetilde\Phi]{r}{s}(\underline{a};\underline{b};x_1,\dots,x_n;q,t)$ is the unique solution in $\mathscr{F}^{(x)}$ of the equation
		\begin{align}
			\(\hyper[\widetilde\M^{(n)}]{r}{}(\underline{a})-\frac{1}{c}\cdot\hyper[\widetilde\L^{(n)}]{}{s}(\underline{b})\)(F(x_1,\dots,x_n))=0, 
		\end{align}
		subject to the initial condition that $F(\bm0_n)=1$, and the stability condition that for each $1\leq m\leq n-1$, $\hyper[\widetilde\Phi]{r}{s}(\underline{a};\underline{b};x_1,\dots,x_m,0,\dots,0;q,t)$ is a solution of the equation
		\begin{align}
			\(\hyper[\widetilde\M^{(m)}]{r}{}(\underline{a})-\frac{1}{c}\hyper[\widetilde\L^{(m)}]{}{s}(\underline{b})\) (F(x_1,\dots,x_m)) = 0.
		\end{align}
	\end{namedtheorem*}
	\begin{namedtheorem*}{Theorem $\widetilde C$}
		Let $c=\frac{a_1\cdots a_r}{qb_1\dots b_s}$ be as above.
		The hypergeometric series  $\hyper[\widetilde\Phi]{r}{s}(\underline{a};\underline{b};x;q,t)$ is the unique solution in $\mathscr F^{(x)}$ of the equation
		\begin{align}
			\(c\cdot \hyper[\widetilde\R]{r}{}(\underline{a})-\hyper[\widetilde\N]{}{s}(\underline{b})\)(F(x))=0,
		\end{align}
		subject to the condition that $F(\bm0_n)=1$, i.e., $C_{(0)}(q,t)=1$.
	\end{namedtheorem*}
	
	For instance, for $\hyper{0}{0}(x)=\prod_i \poch{x_i;q}{\infty}^{-1}$, we have 
	\begin{align*}
		\L= E,\quad \M=\frac{1}{1-q}\frac{1-t^n}{1-t},\quad \N=\square,\quad \R=e_1',
	\end{align*}
	and $c=1/q$.
	
	One can check that for $\hyper[\widetilde\Phi]{0}{0}(x)=\prod_i \poch{x_i;q}{\infty}$, we have 
	\begin{align*}
		\qpartial[1/q]{x_i}\(\prod_i\poch{x_i;q}{\infty}\)=\frac{1}{q-1}\prod_i(x_i;q),
	\end{align*}
	and so 
	\begin{align*}
		\frac{1}{c}\widetilde L\(\prod_i(x_i;q)\)
		&=	q\sum_i A_i(x;1/t)\qpartial[1/q]{x_i} \(\prod_i(x_i;q)\)
		\\&=	\sum_i A_i(x;1/t) \frac{q}{q-1} \prod_i(x_i;q)
		\\&=	\frac{1}{1-1/q} \frac{1-1/t^n}{1-1/t}  \cdot\prod_i(x_i;q)
		=\widetilde M\(\prod_i(x_i;q)\),
	\end{align*}
	and
	\begin{align*}
		\widetilde\N\(\prod_i(x_i;q)\)
		&=	t^{n-1}\sum_{i} x_iA_i(x;1/t)\qpartial[1/q]{x_i}\(\prod_i(x_i;q)\)
		\\&=	t^{n-1}\sum_{i} x_iA_i(x;1/t)\frac{1}{q-1}\prod_i(x_i;q)
		\\&=	\frac{1}{q}\frac{e_1}{1-1/q} \cdot \prod_i(x_i;q) = c\widetilde R\(\prod_i(x_i;q)\)
		.
	\end{align*}
	as desired.
\subsection{Basic hypergeometric series}
	In this subsection, we discuss how the multivariate theory in this paper is related to the univariate setting discussed briefly in \cref{sec:intro}.
	Setting $n=1$ in Kaneko's series \cref{eqn:Kaneko-1}, we have
	\begin{align}
		\hyper[\widetilde\Phi]{r}{s}(\underline{a};\underline{b};z;q,t) = \sum_{k=0}^\infty \((-1)^k q^{\binom{k}{2}}\)^{s+1-r} \frac{\poch{\underline{a};q}{k}}{\poch{\underline{b};q}{k}} \frac{z^k}{\poch{q;q}{k}},
	\end{align}
	which is precisely \cref{eqn:basic}, the modern definition of the univariate basic hypergeometric series $\hyper[\phi]{r}{s}$ (whereas Macdonald's definition \cref{eqn:pFq-Mac-1-new} reduces to the ``old'' definition of basic hypergeometric series, defined without the factor $\((-1)^k q^{\binom{k}{2}}\)^{s+1-r}$; see \cite[Section 1.2]{GR04}). 
	We see that the RHS is independent of $t$.
	
	Let $T_q$ be the $q$-shift operator such that $T_q f(z) = f(qz)$. Write $\Delta_a  = aT_q-1$.
	The operators $e_1',E,\square$ and $D_1$ become
	\begin{align*}
		e_1' = \frac{z}{1-q},	\quad
		E = \frac{1}{q-1}\frac{1}{z}\Delta_1,	\quad
		\square = \frac{1}{q-1}\Delta_1,	\quad
		D_1 = T_q.
	\end{align*}
	Then we have for $l\geq0$, 
	\begin{align*}
		\ad_{-\square}^l(E) = q^{-l} E T_q^l,	\quad
		\ad_{\square}^l(e_1') = e_1' T_q^l,	
	\end{align*}
	and
	\begin{align*}
		\mathcal G_1(u) 
		&= \frac{1}{(1-q)(1-t)} \(1-t\frac{1-u/tD_1}{1-uD_1}\) 
		= \frac{1}{1-q}\sum_{l=0}^\infty u^lT_q^l,	\\
		\mathcal H(u) 
		&=\frac{t}{(1-q)(1-t)} \frac{u-q}{u}\(t^{-1}\frac{1-ut/qD_1}{1-u/qD_1}-1\)+\frac{q}{1-q}\frac{1}{u}
		=\frac{\Delta_1}{q-1} \sum_{l=0}^\infty u^lq^{-l}T_q^l.
	\end{align*}
	Hence, 
	\begin{align}
		\hyper[\L]{}{s}(\underline{b}) &= \sum_{l=0}^s (-1)^le_{l}(\underline{b})q^{-l} E T_q^l =  \frac{(-1)^{s+1}}{1-q}\frac{1}{z}\Delta_1\Delta_{b_1/q}\cdots\Delta_{b_s/q},	\\
		\hyper[\M]{r}{}(\underline{a}) &=	\frac{1}{1-q}\sum_{l=0}^r (-1)^l e_l(\underline{a})T_q^l = \frac{(-1)^r}{1-q} \Delta_{a_1}\cdots\Delta_{a_r},	\\
		\hyper[\N]{}{s}(\underline{b}) &= \frac{\Delta_1}{q-1} \sum_{l=0}^s (-1)^le_l(\underline{b}) q^{-l}T_q^l = \frac{(-1)^{s+1}}{1-q}\Delta_1 \Delta_{b_1/q}\cdots \Delta_{b_s/q},	\\
		\hyper[\R]{r}{}(\underline{a}) &= \sum_{l=0}^r (-1)^le_{l}(\underline{a})e_1' T_q^l =  \frac{(-1)^r}{1-q} z\Delta_{a_1}\cdots \Delta_{a_r}.
	\end{align}
	
	Recall from the last subsection that putting $\sim$ above an operator means inverting the parameters. We then have
	\begin{align*}
		\widetilde{\Delta_a} = \frac{1}{a} T_{1/q}-1 = -\frac{1}{a}T_{1/q}\Delta_a,
	\end{align*}
	and hence,
	\begin{align}
		\hyper[\widetilde\L]{}{s}(\underline{b}) &=  \frac{q^{s+1}}{q-1}\frac{1}{b_1\cdots b_s} \frac{1}{z}T_{1/q}^{s+1}\Delta_1 \Delta_{b_1/q}\cdots\Delta_{b_s/q},	\\
		\hyper[\widetilde\M]{r}{}(\underline{a}) &= \frac{q}{q-1} \frac{1}{a_1\cdots a_r} T_{1/q}^r\Delta_{a_1}\cdots\Delta_{a_r}	\\
		\hyper[\widetilde\N]{}{s}(\underline{b}) &= \frac{q^{s+1}}{q-1}\frac{1}{b_1\cdots b_s}T_{1/q}^{s+1}\Delta_1 \Delta_{b_1/q}\cdots \Delta_{b_s/q},	\\
		\hyper[\widetilde\R]{r}{}(\underline{a}) &= \frac{q}{q-1} \frac{1}{a_1\cdots a_r} zT_{1/q}^r \Delta_{a_1}\cdots \Delta_{a_r}.
	\end{align}
	Recall that $c=\frac{a_1\cdots a_r}{q b_1\cdots b_s}$ and note that $q^{-s-1} T_{q}^{s+1}z = z T_{q}^{s+1}$, then
	\begin{align*}
		c\cdot \hyper[\widetilde\R]{r}{}(\underline{a})-\hyper[\widetilde\N]{}{s}(\underline{b})
		&=	\frac{q^{s+1}}{q-1}\frac{T_{1/q}^{s+1}}{b_1\cdots b_s}\(z \Delta_{a_1}\cdots \Delta_{a_r}T_{q}^{s+1-r} -\Delta_1 \Delta_{b_1/q}\cdots \Delta_{b_s/q}\),
	\end{align*}
	which recovers \cref{eqn:q-difference}.
	As for $\hyper[\widetilde\M]{r}{}(\underline{a})-\frac{1}{c}\cdot\hyper[\widetilde\L]{}{s}(\underline{b})$, we have 
	\begin{align*}
		\hyper[\widetilde\M]{r}{}(\underline{a})-\frac{1}{c}\cdot\hyper[\widetilde\L]{}{s}(\underline{b})
		&=	\frac{q^{s+2}}{q-1} \frac{1}{a_1\cdots a_r}\frac{1}{z} T_{1/q}^{s+1} \(  z\Delta_{a_1}\cdots\Delta_{a_r}T_{q}^{s+1-r} - \Delta_1 \Delta_{b_1/q}\cdots\Delta_{b_s/q}\),
	\end{align*}
	which is equal to $\frac{1}{cz}\cdot \(c\cdot \hyper[\widetilde\R]{r}{}(\underline{a})-\hyper[\widetilde\N]{}{s}(\underline{b})\)$.
	
	In conclusion, the $q$-difference equations in this paper are natural multivariate generalization of the univariate theory.
\section*{Acknowledgments}
	The author is grateful to Siddhartha Sahi for many helpful discussions and valuable guidance.
	This work was partially supported by the Lebowitz Summer Research Fellowship and the SAS Fellowship at Rutgers University.

\printbibliography

@misc {CS-Jack,
	AUTHOR = {Chen, Hong and Sahi, Siddhartha},
	TITLE = {A characterization of {M}acdonald's {J}ack hypergeometric series ${}_pF_q(x;\alpha)$ and ${}_pF_q(x,y;\alpha)$ via differential equations},
	year={2025},
	eprint={2510.10875},
	archivePrefix={arXiv},
	primaryClass={math.CO},
	url={https://arxiv.org/abs/2510.10875}, 
}

@phdthesis{Chen-thesis,
  author = {Hong Chen},
  title  = {Hypergeometric Series Associated with Jack and Macdonald Polynomials},
  school = {Rutgers University},
  year         = {2026},
  address      = {Piscataway, NJ},
  type         = {Ph.D. thesis},
}

@misc{LWYZ-Mac,
      title={Superintegrability for some $(q,t)$-deformed matrix models}, 
      author={Liu, Fan and Wang, Rui and Yang, Jie and Zhao, Wei-Zhong},
      year={2025},
      eprint={2510.18524},
      archivePrefix={arXiv},
      primaryClass={hep-th},
      url={https://arxiv.org/abs/2510.18524}, 
}

@book {Noumi,
	AUTHOR = {Noumi, Masatoshi},
	TITLE = {Macdonald polynomials---commuting family of {$q$}-difference
	operators and their joint eigenfunctions},
	SERIES = {SpringerBriefs in Mathematical Physics},
	VOLUME = {50},
	PUBLISHER = {Springer, Singapore},
	YEAR = {[2023] \copyright 2023},
	PAGES = {viii+132},
	ISBN = {978-981-99-4586-3; 978-981-99-4587-0},
	MRCLASS = {33D52 (33D45)},
	MRNUMBER = {4647625},
	DOI = {10.1007/978-981-99-4587-0},
	URL = {https://doi.org/10.1007/978-981-99-4587-0},
}

@book {GR04,
	AUTHOR = {Gasper, George and Rahman, Mizan},
	TITLE = {Basic hypergeometric series},
	SERIES = {Encyclopedia of Mathematics and its Applications},
	VOLUME = {96},
	EDITION = {Second},
	PUBLISHER = {Cambridge University Press, Cambridge},
	YEAR = {2004},
	PAGES = {xxvi+428},
	ISBN = {0-521-83357-4},
	MRCLASS = {33Dxx (05A30 05E35 33-01 33-02)},
	MRNUMBER = {2128719},
	MRREVIEWER = {Shaun\ Cooper},
	DOI = {10.1017/CBO9780511526251},
	URL = {https://doi.org/10.1017/CBO9780511526251},
}

@book {AAR99,
	AUTHOR = {Andrews, George E. and Askey, Richard and Roy, Ranjan},
	TITLE = {Special functions},
	SERIES = {Encyclopedia of Mathematics and its Applications},
	VOLUME = {71},
	PUBLISHER = {Cambridge University Press, Cambridge},
	YEAR = {1999},
	PAGES = {xvi+664},
	ISBN = {0-521-62321-9; 0-521-78988-5},
	MRCLASS = {33-01 (33-02)},
	MRNUMBER = {1688958},
	MRREVIEWER = {Bruce\ C.\ Berndt},
	DOI = {10.1017/CBO9781107325937},
	URL = {https://doi.org/10.1017/CBO9781107325937},
}

@article {Yan92,
	AUTHOR = {Yan, Zhi Min},
	TITLE = {A class of generalized hypergeometric functions in several
	variables},
	JOURNAL = {Canad. J. Math.},
	FJOURNAL = {Canadian Journal of Mathematics. Journal Canadien de
	Math\'ematiques},
	VOLUME = {44},
	YEAR = {1992},
	NUMBER = {6},
	PAGES = {1317--1338},
	ISSN = {0008-414X,1496-4279},
	MRCLASS = {33C70 (32A99)},
	MRNUMBER = {1192421},
	MRREVIEWER = {A.\ K.\ Agarwal},
	DOI = {10.4153/CJM-1992-079-x},
	URL = {https://doi.org/10.4153/CJM-1992-079-x},
}

@article{Con63,
	title = {Some non-central distribution problems in multivariate analysis},
	author = {Constantine, A. G.},
	date = {1963-12},
	journaltitle = {The Annals of Mathematical Statistics},
	shortjournal = {Ann. Math. Statist.},
	volume = {34},
	number = {4},
	pages = {1270--1285},
	issn = {0003-4851},
	doi = {10.1214/aoms/1177703863},
	url = {http://projecteuclid.org/euclid.aoms/1177703863},
	urldate = {2025-04-05},
	langid = {english},
}

@article {M70,
	AUTHOR = {Muirhead, R. J.},
	TITLE = {Systems of partial differential equations for hypergeometric
	functions of matrix argument},
	JOURNAL = {Ann. Math. Statist.},
	FJOURNAL = {Annals of Mathematical Statistics},
	VOLUME = {41},
	YEAR = {1970},
	PAGES = {991--1001},
	ISSN = {0003-4851},
	MRCLASS = {33.20 (62.40)},
	MRNUMBER = {264799},
	MRREVIEWER = {I.\ Olkin},
	DOI = {10.1214/aoms/1177696975},
	URL = {https://doi.org/10.1214/aoms/1177696975},
}

@book {Erd81,
	AUTHOR = {Erd\'elyi, Arthur and Magnus, Wilhelm and Oberhettinger, Fritz
	and Tricomi, Francesco G.},
	TITLE = {Higher transcendental functions. {V}ol. {I}},
	PUBLISHER = {Robert E. Krieger Publishing Co., Inc., Melbourne, FL},
	YEAR = {1981},
	PAGES = {xiii+302},
	ISBN = {0-89874-069-X},
	MRCLASS = {33-02 (01A75)},
	MRNUMBER = {698779},
}

@article {Kan93,
	AUTHOR = {Kaneko, Jyoichi},
	TITLE = {Selberg integrals and hypergeometric functions associated with
	{J}ack polynomials},
	JOURNAL = {SIAM J. Math. Anal.},
	FJOURNAL = {SIAM Journal on Mathematical Analysis},
	VOLUME = {24},
	YEAR = {1993},
	NUMBER = {4},
	PAGES = {1086--1110},
	ISSN = {0036-1410},
	MRCLASS = {33C80},
	MRNUMBER = {1226865},
	MRREVIEWER = {Kevin\ W. J. Kadell},
	DOI = {10.1137/0524064},
	URL = {https://doi.org/10.1137/0524064},
}

@article {CM72,
	AUTHOR = {Constantine, A. G. and Muirhead, R. J.},
	TITLE = {Partial differential equations for hypergeometric functions of
	two argument matrices},
	JOURNAL = {J. Multivariate Anal.},
	FJOURNAL = {Journal of Multivariate Analysis},
	VOLUME = {2},
	YEAR = {1972},
	PAGES = {332--338},
	ISSN = {0047-259X},
	MRCLASS = {33A30 (62H10)},
	MRNUMBER = {333262},
	MRREVIEWER = {A.\ M.\ Mathai},
	DOI = {10.1016/0047-259X(72)90020-6},
	URL = {https://doi.org/10.1016/0047-259X(72)90020-6},
}

@article {Fujikoshi,
	AUTHOR = {Fujikoshi, Yasunori},
	TITLE = {Partial differential equations for hypergeometric functions
	{$\sb{3}F\sb{2}$} of matrix argument},
	JOURNAL = {Canad. J. Statist.},
	FJOURNAL = {The Canadian Journal of Statistics. La Revue Canadienne de
	Statistique},
	VOLUME = {3},
	YEAR = {1975},
	NUMBER = {2},
	PAGES = {153--163},
	ISSN = {0319-5724,1708-945X},
	MRCLASS = {62H10 (33A30)},
	MRNUMBER = {400540},
	DOI = {10.2307/3315276},
	URL = {https://doi.org/10.2307/3315276},
}

@misc{Mac-HG,
	title={Hypergeometric functions {I}}, 
	author={Ian G. Macdonald},
	year={2013},
	eprint={1309.4568},
	archivePrefix={arXiv},
	primaryClass={math.CA},
	url={https://arxiv.org/abs/1309.4568}, 
}

@misc{Mac-HG-2,
	title={Hypergeometric functions {II} ($q$-analogues)}, 
	author={Ian G. Macdonald},
	year={2013},
	eprint={1309.5208},
	archivePrefix={arXiv},
	primaryClass={math.CA},
	url={https://arxiv.org/abs/1309.5208}, 
}

@book {AskeyBateman,
	TITLE = {Encyclopedia of special functions: the {A}skey-{B}ateman
	project. {V}ol. 2. {M}ultivariable special functions},
	EDITOR = {Koornwinder, Tom H. and Stokman, Jasper V.},
	PUBLISHER = {Cambridge University Press, Cambridge},
	YEAR = {2021},
	PAGES = {xii+427},
	ISBN = {978-1-107-00373-6; 978-1-108-88244-6},
	MRCLASS = {33-02 (05Axx 33C20 33C52 33C65 33C80)},
	MRNUMBER = {4421384},
	DOI = {10.1017/9780511777165},
	URL = {https://doi.org/10.1017/9780511777165},
}

@article {War05,
	AUTHOR = {Warnaar, S. Ole},
	TITLE = {{$q$}-{S}elberg integrals and {M}acdonald polynomials},
	JOURNAL = {Ramanujan J.},
	FJOURNAL = {Ramanujan Journal. An International Journal Devoted to the
	Areas of Mathematics Influenced by Ramanujan},
	VOLUME = {10},
	YEAR = {2005},
	NUMBER = {2},
	PAGES = {237--268},
	ISSN = {1382-4090,1572-9303},
	MRCLASS = {33D05 (33D52 33D60)},
	MRNUMBER = {2195565},
	MRREVIEWER = {Song\ Heng\ Chan},
	DOI = {10.1007/s11139-005-4849-7},
	URL = {https://doi.org/10.1007/s11139-005-4849-7},
}

@article {FW08,
	AUTHOR = {Forrester, Peter J. and Warnaar, S. Ole},
	TITLE = {The importance of the {S}elberg integral},
	JOURNAL = {Bull. Amer. Math. Soc. (N.S.)},
	FJOURNAL = {American Mathematical Society. Bulletin. New Series},
	VOLUME = {45},
	YEAR = {2008},
	NUMBER = {4},
	PAGES = {489--534},
	ISSN = {0273-0979,1088-9485},
	MRCLASS = {33-02 (33-03)},
	MRNUMBER = {2434345},
	MRREVIEWER = {Shaun\ Cooper},
	DOI = {10.1090/S0273-0979-08-01221-4},
	URL = {https://doi.org/10.1090/S0273-0979-08-01221-4},
}

@article {BF97,
	AUTHOR = {Baker, T. H. and Forrester, P. J.},
	TITLE = {The {C}alogero-{S}utherland model and generalized classical
	polynomials},
	JOURNAL = {Comm. Math. Phys.},
	FJOURNAL = {Communications in Mathematical Physics},
	VOLUME = {188},
	YEAR = {1997},
	NUMBER = {1},
	PAGES = {175--216},
	ISSN = {0010-3616,1432-0916},
	MRCLASS = {33C45 (33C50 33C60 81V70)},
	MRNUMBER = {1471336},
	MRREVIEWER = {Laurent\ Habsieger},
	DOI = {10.1007/s002200050161},
	URL = {https://doi.org/10.1007/s002200050161},
}

@incollection {BF99,
	AUTHOR = {Baker, T. H. and Forrester, P. J.},
	TITLE = {Transformation formulas for multivariable basic hypergeometric
	series},
	booktitle = {Methods and Applications of Analysis},
	VOLUME = {6},
	YEAR = {1999},
	NUMBER = {2},
	PAGES = {147--164},
	ISSN = {1073-2772,1945-0001},
	MRCLASS = {33D52 (33D80 39A13)},
	MRNUMBER = {1803887},
	MRREVIEWER = {David\ M.\ Bressoud},
	DOI = {10.4310/MAA.1999.v6.n2.a2},
	URL = {https://doi.org/10.4310/MAA.1999.v6.n2.a2},
}

@article {Herz,
	AUTHOR = {Herz, Carl S.},
	TITLE = {Bessel functions of matrix argument},
	JOURNAL = {Ann. of Math. (2)},
	FJOURNAL = {Annals of Mathematics. Second Series},
	VOLUME = {61},
	YEAR = {1955},
	PAGES = {474--523},
	ISSN = {0003-486X},
	MRCLASS = {33.0X},
	MRNUMBER = {69960},
	MRREVIEWER = {A.\ Erd\'elyi},
	DOI = {10.2307/1969810},
	URL = {https://doi.org/10.2307/1969810},
}

@article {Las98,
	AUTHOR = {Lassalle, Michel},
	TITLE = {Coefficients binomiaux g\'en\'eralis\'es et polyn\^omes de
	{M}acdonald},
	JOURNAL = {J. Funct. Anal.},
	FJOURNAL = {Journal of Functional Analysis},
	VOLUME = {158},
	YEAR = {1998},
	NUMBER = {2},
	PAGES = {289--324},
	ISSN = {0022-1236,1096-0783},
	MRCLASS = {33D67 (05A30 33D45 33D80)},
	MRNUMBER = {1648471},
	MRREVIEWER = {Jiang\ Zeng},
	DOI = {10.1006/jfan.1998.3281},
	URL = {https://doi.org/10.1006/jfan.1998.3281},
}

@article {HHL05,
	AUTHOR = {Haglund, J. and Haiman, M. and Loehr, N.},
	TITLE = {A combinatorial formula for {M}acdonald polynomials},
	JOURNAL = {J. Amer. Math. Soc.},
	FJOURNAL = {Journal of the American Mathematical Society},
	VOLUME = {18},
	YEAR = {2005},
	NUMBER = {3},
	PAGES = {735--761},
	ISSN = {0894-0347,1088-6834},
	MRCLASS = {05E05},
	MRNUMBER = {2138143},
	MRREVIEWER = {Frank\ Sottile},
	DOI = {10.1090/S0894-0347-05-00485-6},
	URL = {https://doi.org/10.1090/S0894-0347-05-00485-6},
}

@book {Mac15,
	AUTHOR = {Macdonald, I. G.},
	TITLE = {Symmetric functions and {H}all polynomials},
	SERIES = {Oxford Classic Texts in the Physical Sciences},
	EDITION = {Second},
	EDITION = {paperback},
	PUBLISHER = {The Clarendon Press, Oxford University Press, New York},
	YEAR = {2015},
	PAGES = {xii+475},
	ISBN = {978-0-19-873912-8},
	MRCLASS = {05E05 (01A75 05-02 20C30 20C33 20K01 33C80 33D80)},
	MRNUMBER = {3443860},
}

@article {Kan96,
	AUTHOR = {Kaneko, Jyoichi},
	TITLE = {{$q$}-{S}elberg integrals and {M}acdonald polynomials},
	JOURNAL = {Ann. Sci. \'Ecole Norm. Sup. (4)},
	FJOURNAL = {Annales Scientifiques de l'\'Ecole Normale Sup\'erieure.
	Quatri\`eme S\'erie},
	VOLUME = {29},
	YEAR = {1996},
	NUMBER = {5},
	PAGES = {583--637},
	ISSN = {0012-9593},
	MRCLASS = {33C80 (05A30 05E05 33D20 33D45)},
	MRNUMBER = {1399617},
	URL = {http://www.numdam.org/item?id=ASENS_1996_4_29_5_583_0},
}

@article {NS14,
	AUTHOR = {Nazarov, M. L. and Sklyanin, E. K.},
	TITLE = {Macdonald operators at infinity},
	JOURNAL = {J. Algebraic Combin.},
	FJOURNAL = {Journal of Algebraic Combinatorics. An International Journal},
	VOLUME = {40},
	YEAR = {2014},
	NUMBER = {1},
	PAGES = {23--44},
	ISSN = {0925-9899,1572-9192},
	MRCLASS = {05E05},
	MRNUMBER = {3226816},
	MRREVIEWER = {Meesue\ Yoo},
	DOI = {10.1007/s10801-013-0477-2},
	URL = {https://doi.org/10.1007/s10801-013-0477-2},
}

@misc{CS24,
	title={Interpolation polynomials, binomial coefficients, and symmetric function inequalities}, 
	author={Chen, Hong and Sahi, Siddhartha},
	year={2024},
	eprint={2403.02490},
	archivePrefix={arXiv},
	primaryClass={math.CO},
	url={https://arxiv.org/abs/2403.02490}, 
}
\end{document}